\documentclass[11pt]{amsart} 

\pdfoutput=1

\usepackage[utf8]{inputenc}
\usepackage[english]{babel}

\usepackage{amsmath}
\usepackage{amsfonts}
\usepackage{amssymb}
\usepackage{amsthm}

\usepackage{mathtools}        
\usepackage{stmaryrd}         

\usepackage[inline]{enumitem}
\setlist{itemsep=1pt,topsep=3pt,partopsep=0pt,leftmargin=30pt}
\SetEnumitemKey{(1)}{label=\textrm{(\arabic*)}}
\SetEnumitemKey{(i)}{label=\textrm{(\roman*)}}
\SetEnumitemKey{(I)}{label=\textrm{(\Roman*)}}
\SetEnumitemKey{(a)}{label=\textrm{(\alph*)}}
\SetEnumitemKey{(A)}{label=\textrm{(\Alph*)}}

\SetEnumitemKey{1}{label=\textrm{\arabic*}}
\SetEnumitemKey{i}{label=\textrm{\roman*}}
\SetEnumitemKey{I}{label=\textrm{\Roman*}}
\SetEnumitemKey{a}{label=\textrm{\alph*}}
\SetEnumitemKey{A}{label=\textrm{\Alph*}}

\usepackage[a4paper,left=3.3cm,top=4cm,right=3.3cm,bottom=3.3cm]{geometry}
\usepackage[protrusion=true,expansion=true]{microtype}

\usepackage[
            colorlinks=true,
            linkcolor=blue]{hyperref}

\usepackage[capitalise]{cleveref}     

\usepackage{csquotes}          
\usepackage[backend = bibtex,
            language = english ,
            style    = numeric-comp ,
            firstinits = true,
            isbn = false,
            url = true,
            doi = false,
            sorting = nyt,
            backref = false,
            eprint = false
            ]{biblatex}
\DeclareNameAlias{sortname}{last-first}
\renewbibmacro{in:}{%
  \ifentrytype{article}{}{%
  \printtext{\bibstring{in}\intitlepunct}}}
\AtEveryBibitem{%
  \ifentrytype{article}
  {%
    \clearfield{url}%
    \clearfield{urldate}%
    \clearfield{urlyear}%
  }
  {}%
}

\AtEveryBibitem{%
  \ifentrytype{incollection}
  {%
    \clearfield{url}%
    \clearfield{urldate}%
    \clearfield{urlyear}%
  }
  {}%
}

\AtEveryBibitem{%
  \ifentrytype{unpublished}
  {%
    \clearfield{urldate}%
    \clearfield{urlyear}%
  }
  {}%
}

\usepackage{todonotes}

\newtheorem{thm}{Theorem}[section]
\newtheorem{cor}[thm]{Corollary}
\newtheorem{lem}[thm]{Lemma}
\newtheorem{prop}[thm]{Proposition}

\newtheorem{qst}{Question}
\theoremstyle{definition}
\newtheorem{defn}[thm]{Definition}
\theoremstyle{remark}
\newtheorem{rem}[thm]{Remark}

\newcommand{\defin}[1]{\emph{#1}}               
\newcommand{\comment}[1]{{}}                    

\newcommand{\ie}{i.e.,~}
\newcommand{\eg}{e.g.,~}

\newcommand\NN{\mathbb{N}}
\newcommand\ZZ{\mathbb{Z}}

\newcommand\RR{\mathbb{R}}

\newcommand\GG{\mathcal{G}}
\newcommand\FF{\mathcal{F}}
\newcommand\HH{\mathcal{H}}
\newcommand\St{\mathcal{S}}

\newcommand\BB{\mathcal{B}}

\newcommand{\Imp}{\ \Rightarrow \ }              
\newcommand{\nImp}{\ \nRightarrow \ }            
\newcommand{\Biimp}{\ \Leftrightarrow \ }        

\newcommand{\propsty}[1]{\mathsf{#1}}

\newcommand{\propi}{\propsty{P}}
\newcommand{\fg}{\propsty{fg}}                         
\newcommand{\fp}{\propsty{fp}}                         
\newcommand{\der}{\propsty{d}}                         
\newcommand{\by}{\operatorname{-by-}}
\newcommand{\isom}{\simeq}     

\newcommand{\trivial}{1}                              
\newcommand{\normaleq}{ \trianglelefteqslant }        

\newcommand{\Pres}[2]{\left\langle  #1 \, \left| \ #2 \, \right. \hspace{-2pt} \right\rangle} 
\newcommand{\pres}[2]{\langle  \, #1 \! \mid \! #2 \, \rangle} 
\newcommand{\gen}[1]{\langle  #1 \rangle}  
\newcommand{\normalcl}[1]{\left \langle \! \left \langle #1 \right \rangle \! \right \rangle} 

\newcommand{\Betti}[1]{\operatorname{b}\!\left(#1\right)} 
\newcommand{\betti}[1]{\operatorname{b}(#1)} 

\DeclareMathOperator*{\Freeprod}{\scalebox{2}{\raisebox{-0.3ex}{$\ast\hspace{-0.3pt}$}}}  
\newcommand{\freeprod}{*} 
\newcommand{\rk}{\operatorname{rk}}

\newcommand{\ab}{^{\operatorname{ab}}}  
\newcommand{\abf}{^{\operatorname{ab}}_{*}}  
\newcommand{\piab}{\pi^{\operatorname{ab}}}  
\newcommand{\piabf}{\pi^{\operatorname{ab}}_{*}} 

\newcommand{\onto}{\twoheadrightarrow}
\newcommand{\into}{\hookrightarrow}
\newcommand{\im}{\operatorname{im}}                 
\newcommand{\id}{\operatorname{id}}                 
\newcommand{\Hom}{\operatorname{Hom}}               
\newcommand{\Aut}{\operatorname{Aut}}               
\newcommand{\End}{\operatorname{End}}               
\newcommand{\GL}{\operatorname{GL}}                 
\newcommand{\Inn}{\operatorname{Inn}}               
\newcommand{\Out}{\operatorname{Out}}               

\newcommand{\CP}{\operatorname{CP}}
\newcommand{\SCP}{\textstyle\frac{1}{2}\!\operatorname{CP}}
\newcommand{\DCP}[1]{\operatorname{CP}_{\!\scriptscriptstyle{#1}}}
\newcommand{\DSCP}[1]{\textstyle\frac{1}{2}\!\operatorname{CP}_{\!\scriptscriptstyle{#1}}}
\newcommand{\IP}{\operatorname{IP}}
\newcommand{\MP}{\operatorname{MP}}
\newcommand{\EP}{\operatorname{EP}}

\newcommand{\BS}{\operatorname{BS}}

\newcommand{\YES}{\texttt{YES}}
\newcommand{\NO}{\texttt{NO}}

\newcommand{\bns}[1]{\Sigma(#1)}                    
\newcommand{\sph}[1]{\operatorname{S}(#1)}          

\newcommand{\fpa}[1]{[#1]_{\fp}}   
\newcommand{\fga}[1]{[#1]_{\fg}}   

\begin{document}

\title[]{Algorithmic recognition of infinite cyclic extensions}

\author{
	Bren Cavallo
}
\address{Department of Mathematics, CUNY Graduate Center,  City University of New York.}
\email{bcavallo@gradcenter.cuny.edu}

\author{
	Jordi Delgado 
}
\address{Departament de Matemàtiques, Universitat Politècnica de Catalunya.}
\email{jorge.delgado@upc.edu}

\author{
	Delaram Kahrobaei
}

\address{CUNY  Graduate  Center,  PhD  Program  in  Computer  Science
and NYCCT, Mathematics Department, City University of New York.}
\email{dkahrobaei@gc.cuny.edu}

\author{
	Enric Ventura
}
\address{Departament de Matemàtiques, Universitat Politècnica de Catalunya.}
\email{enric.ventura@upc.edu}%

\date{\today}



\begin{abstract}
	We prove that one cannot algorithmically decide whether a finitely presented $\ZZ$-extension admits a finitely generated base group, and we use this fact to prove the undecidability of the BNS invariant. Furthermore, we show the equivalence between the isomorphism problem within the subclass of unique $\ZZ$-extensions, and the semi-conjugacy problem for deranged outer automorphisms.
\end{abstract}

\maketitle

In the present paper, we study algorithmic problems about recognition of certain algebraic properties among some families of group extensions.
Indeed, we see that yet for the relatively easy family of $\ZZ$-extensions
one can find positive and negative results, \ie both solvable and unsolvable ``recognition problems''.

For example, we prove that one cannot algorithmically decide whether a finitely presented $\ZZ$-extension admits a finitely generated base group. Even when the extension has a unique possible base group, it is not decidable in general whether this particular base group is finitely generated or not. As a consequence, we prove general undecidability for the Bieri--Neumann--Strebel invariant: there is no algorithm which, on input a finite presentation of a group $G$ and a character $\chi \colon G\to \RR$, decides whether $[\chi]$ belongs to the BNS invariant of $G$, $[\chi]\in \bns{G}$, or not.
Although this result seems quite natural,
since this geometric invariant has long been agreed to be hard to compute in general (see for example~\cite{meier_bieri-neumann-strebel_1995,
	papadima_bierineumannstrebelrenz_2010, koban_bns-invariant_2015, koban_bierineumannstrebel_2014}), as far as we know, its undecidability does not seem to be contained in the literature. Following our study of recognition properties, we finally consider the isomorphism problem in certain classes of unique $\ZZ$-extensions, and prove that it is equivalent to the semi-conjugacy problem for the corresponding deranged outer automorphisms (see details in \cref{sec:isom}).

\medskip
The structure of the paper is as follows. In \cref{recog} we state the recognition problems we are interested in. In \cref{sec:extensions} we introduce
the most general framework for our study:
(finitely presented) $\ZZ^r$-extensions (denoted $* \by \ZZ^r$), unique $\ZZ^r$-extensions (denoted $! \by \ZZ^r$), as well as the subfamily of $\fg \by \ZZ^r$ groups, and will investigate the above problems for them. In Sections~\ref{sec:Z-extensions} and~\ref{unique} we focus on the case $r=1$ (\ie infinite cyclic extensions) which will be the main target of the paper. The central result in Section~\ref{sec:undecidability} is \cref{thm:fg-by-Z is undecidable}, showing that the membership problem for $\fg \by \ZZ$ (among other similar families) is undecidable, even within the class $! \by \ZZ$. As an application, \cref{sec:BNS} contains the undecidability of the BNS invariant (\cref{thm:BNS is not decidable for Betti 1}). In \cref{sec:searching standard presentations} we search for ``standard presentations'' of $\fg \by \ZZ$ groups (\cref{prop:enumerate standard presentations of fg-by-Z}). Finally, in \cref{sec:isom} we characterize the isomorphism problem in the subclass of unique $\ZZ$-extensions by means of the so-called semi-conjugacy problem (a weakened version of the standard conjugacy problem) for deranged outer automorphisms (\cref{thm:main-iso}).

\section{Algorithmic recognition of groups}\label{recog}

Algorithmic behavior of groups has been a very fundamental concern in Combinatorial and
Geometric Group Theory since the very beginning of this branch of Mathematics in the
early 1900's. The famous three problems stated by Max Dehn in 1911 are prototypical examples of
this fact: the Word,  Conjugacy, and  Isomorphism problems have been very influential
in the literature along these last hundred years. Today, these problems (together with
a great and growing collection of variations) are the center of what is known as
Algorithmic Group Theory.

Dehn's Isomorphism Problem is probably the paradigmatic example of what is popularly
understood as ``algorithmic recognition of groups". Namely, let $\GG_\fp$ be the family
of finite presentations of groups. Then (with the usual abuse of notation
of
denoting in the same
way a presentation and the presented group):
\begin{itemize}
	\item \emph{Isomorphism problem} [\,$\IP$\,]: given two finite presentations $G_1,
	      G_2\in \GG_\fp$, decide whether they present isomorphic groups, $G_1 \simeq G_2$,
	      or not.
\end{itemize}

It is well known that, in this full generality, Dehn's Isomorphism Problem is
unsolvable, see for example~\cite{miller_iii_decision_1992}. So, a natural next step is to
study what happens when we restrict the inputs to a certain subfamily $\HH\subseteq
\GG_\fp$:
\begin{itemize}
	\item \emph{Isomorphism problem within $\HH$} [\,$\IP(\HH)\,]$: given two finite
	      presentations $H_1, H_2\in \HH$, decide whether they present isomorphic groups,
	      $H_1\simeq H_2$, or not.
\end{itemize}
Since we are interested in groups, we will only consider families of presentations
closed by isomorphism; in this way, the problems considered are actually about groups
(although represented by finite presentations). The literature is full of results
solving the isomorphism problem for more and more such subfamilies $\HH$ of $\GG_\fp$,
or showing its unsolvability even when restricted to smaller and smaller subfamilies
$\HH$.

\medskip

Another recognition aspect is that of deciding whether a given group satisfies certain
property, \ie whether it belongs to a certain previously defined family. For two
arbitrary subfamilies $\HH, \GG\subseteq \GG_\fp$, we define the:
\begin{itemize}
	\item \emph{(Family) Membership problem for $\HH$ within $\GG$}
	      [\,$\MP_{\GG}(\HH)$\,]: given a finite presentation~${G\in \GG}$, decide whether
	      $G \in \HH$ or not.
\end{itemize}
If $\HH \subseteq \GG$ and $\MP_{\GG}(\HH)$ is decidable we will also say that the
inclusion $\HH \subseteq \GG$ is decidable. When the considered ambient family is the
whole family of finitely presented groups (\ie $\GG = \GG_{\fp}$) we will usually omit
any reference to it and simply talk about the membership problem for $\HH$, denoted
$\MP(\HH)$.

A classic undecidability result due to
Adian~\cite{adian_unsolvability_1957,adian_finitely_1957} and
Rabin~\cite{rabin_recursive_1958} (see also~\cite{miller_iii_decision_1992}) falls into
this scheme. Namely, when $\HH$ is a \emph{Markov} subfamily (\ie a nonempty subfamily
$\emptyset \neq \HH \subseteq \GG_\fp$ such that the subgroups of groups in $\HH$ do
not completely cover~$\GG_\fp$), then $\MP(\HH)$ is not decidable. This turns out to
include the impossibility of deciding membership for countless well known families of
finitely presented groups~(\eg trivial, finite, abelian, nilpotent, solvable, free,
torsion-free, simple, automatic, etc).

Note that $\MP(\HH)$ being decidable is the same as saying that $\HH$ is a recursive
set of finite presentations. And, even when $\MP(\HH)$ is not decidable, we can still
ask for a recursive enumeration of the elements in $\HH$:
\begin{itemize}
	\item \emph{(Family) Enumeration problem for $\HH$} [\,$\EP(\HH)$\,]: enumerate all
	      the elements in $\HH$.
\end{itemize}

In many cases the considered subfamily $\HH\subseteq \GG_\fp$ entails a concept of
``good" or ``standard" presentation $\St \subseteq \HH$, for the groups presented. For
example, if $\HH$ is the family of (finite presentations for) braid groups $\{ B_n \mid
n\geqslant 2\}$, we can define the set of \emph{standard} presentations $\St \subseteq
\HH$ to be those of the form
\begin{equation*}
	\Pres{\sigma_1,\ldots,\sigma_{n-1}}{
		\begin{array}{cc}
			\sigma_i \sigma_j = \sigma_j \sigma_i,                   & |i-j|>1 \\
			\sigma_i \sigma_j \sigma_i = \sigma_j \sigma_i \sigma_j, & |i-j|=1
		\end{array}
		}\, ;
\end{equation*}
in this case, the family enumeration problem for $\St$ consists, on input an arbitrary
finite presentation presenting a braid group, to compute its (unique) standard one for it, \ie to recognize the number of strands $n$.

\section{Group extensions}\label{sec:extensions}

Let $G$ and $Q$ be arbitrary groups. We say that $G$ is a \defin{group extension by
	$Q$} (or a \defin{$Q$-extension}) if $G$ can be homomorphically mapped onto $Q$, \ie if
there exists a normal subgroup $H\normaleq G$ such that the quotient $G/H$ is
isomorphic to $Q$. Of course, this situation gives rise to the short exact sequence
\begin{equation*}
	1 \to H\to G\to Q \to 1
\end{equation*}
for some group $H$, and we will
say that $G$ is \emph{$* \by Q$}. One
also says that such $H$ is a \defin{base group} for the extension, and that $G$ is an
\defin{extension of $H$ by~$Q$}; so, if we want to specify the base group we will say
that $G$ is~\defin{$H \by Q$}.

We remark that a given group extension by $Q$ may admit many, even non isomorphic,
different base groups (see~\cref{cor:nonisomorphic base groups}).

If $H$ is the only (as subset) normal subgroup of $G$ with quotient $G/H$ isomorphic
to~$Q$, then we say that the $Q$-extension is \defin{unique}, or that $G$ is a
\defin{unique extension by~$Q$}; in the same vein as before, we will
say that $G$ is \defin{$! \by Q$} (or that $G$~is~\defin{$!H \by Q$}, if we want to specify
who is the unique normal subgroup).

It will be convenient to extend this notation allowing to replace the groups $H$ and
$Q$ by any group property (which we will usually write in $\mathsf{sans}$ typeface).
Concretely, given two properties of groups, $\propi_1$, $\propi_2$, we say that a group
$G$ is \defin{$\propi_1 \by \propi_2$} (resp., \emph{$!\propi_1 \by \propi_2$}) if it is
$H \by Q$ (resp., $!H \by Q$) for certain groups $H$ satisfying $\propi_1$, and $Q$
satisfying $\propi_2$. In this way we can easily refer to families of group extensions
in terms of the behavior of their base and quotient groups. So, for example, a group
$G$ is \emph{$\fg \by \ZZ^r$} if it is $H \by \ZZ^r$ for some finitely generated group
$H$. And it is \emph{$!\fg \by \ZZ^r$} if it is $!H \by \ZZ^r$ for some finitely
generated group $H$; \ie if it has a unique normal subgroup with quotient isomorphic to
$\ZZ^r$, which happens to be finitely generated (not to be confused with $G$ having a
unique finitely generated normal subgroup whose quotient is isomorphic to $\ZZ^r$---and possibly some others infinitely generated as well).

When we need to add extra assumptions (\ie satisfying some property $\propi$) on the
elements of certain family $\GG$, we will denote the new family $[\GG]_{\propi}$.
For example, $[\propsty{abelian}]_{\fg}$ denotes the family of finitely generated
abelian groups, while $\fpa{*\by\ZZ}$ denotes the family of finitely presented
extensions by $\ZZ$.

%
%
%

\medskip

Recognizing $\ZZ^r$-extensions (more concretely, solving the membership problem for the
families $* \by \ZZ^r$ and $! \by \ZZ^r$) turns out to be very easy. Recall that, for a
finitely generated group $G$, its abelianization is always of the form $G\ab =\ZZ^r
\oplus T$, where $r$ is a nonnegative integer and $T$ is a finite abelian group (both
canonically determined by $G$ and called, respectively, the \defin{(first) Betti
	number} of $G$, denoted~$r = \betti{G}$, and the \defin{abelian torsion} of $G$).

Clearly, for a finitely generated group $G$, $\betti{G}$ is the maximum rank for a
free-abelian quotient of $G$. Hence, we have the following straightforward
characterizations.

\begin{lem}\label{lem:Betti1 and rank of a free-abelian quotient}
	Let $G$ be a finitely generated group, and $k$ a nonnegative integer. Then,
	\begin{enumerate}[(i)]
		\item $G$ is $* \by \ZZ^k$ if and only if $\betti{G}\geq k$;
		      \item\label{item:Betti1 = max rank of a free-abelian quotient} $G$ is $! \by \ZZ^k$
		      if and only if $\betti{G}=k$. \qed
	\end{enumerate}
\end{lem}

Note that one can easily compute the Betti number of any group given by a finite
presentation: just abelianize it (\ie add as relators the commutators of any pair of
generators in the presentation) and then apply the Classification Theorem for finitely
generated abelian groups, which is clearly algorithmic. Thus, Lemma~\ref{lem:Betti1 and
rank of a free-abelian quotient} immediately implies the decidability of the membership
problem for these families of groups.

\begin{cor}\label{cor:MP(*-by-Z) MP(!-by-Z) decidable}
	For every $k\geqslant 0$, the membership problem for the families $* \by \ZZ^k$ and $!
	\by \ZZ^k$ is decidable; \ie there exists an algorithm which takes any finite
	presentation as input and decides whether the presented group is~$* \by \ZZ^k$
	(resp.,~$! \by \ZZ^k$) or not. \qed
\end{cor}

Let us denote by $\piab \colon G\twoheadrightarrow G\ab$ the abelianization map, by
$\piabf \colon G\twoheadrightarrow \ZZ^{\betti{G}}$ the abelianization map followed by
the canonical projection onto the free-abelian part~$\ZZ^{\betti{G}}$, and let us also
write $G\abf = \piabf(G)$. We collect here some elementary properties of the first
Betti number which will be useful later.

\begin{lem}\label{lem:Betti1 properties}
	Let $G$ be a finitely generated group. Then,
	\begin{enumerate}[(i)]
		\item $\betti{G}=\betti{G\ab}=\betti{G\abf}=\rk{G \abf}\leq \rk(G\ab)$, with equality
		      if and only if $G\ab$ is free-abelian;
		\item for every subgroup $H\leqslant G\ab$, $\betti{G\ab/H}=\betti{G\ab}- \betti{H}$;
		\item if $G\ab =H_1 \oplus \cdots \oplus H_k$, then $\betti{G\ab}=\betti{H_1}+\cdots
		      +\betti{H_k}$. \qed
	\end{enumerate}
\end{lem}

In addition, note that the kernels of both $\piab$ and $\piabf$ are fully characteristic
subgroups of~$G$~(\ie invariant under endomorphisms of $G$). Hence, every $\phi \in
\End(G)$ (resp., every $\phi \in \Aut (G)$) canonically determines endomorphisms
$\phi\ab \in \End(G\ab )$ and ${\phi\abf \in \End(G\abf )}$ (resp., automorphisms
$\phi\ab \in \Aut(G\ab )$ and $\phi\abf \in \Aut(G\abf )$). Of course, after choosing
an abelian basis for $\ZZ^r$, where $r=\betti{G}$, $\phi\abf$ can be thought of as an
$r\times r$ square matrix over the integers (with determinant $\pm 1$ if $\phi$ is an
automorphism of $G$). In the following section we will relate certain properties of a
$\ZZ$-extension with properties of its defining automorphism $\phi$.

\section{$\ZZ$-extensions}\label{sec:Z-extensions}

We will concentrate now on infinite cyclic extensions, concretely in the family $* \by
\ZZ$ and its subfamily $! \by \ZZ$. Let us describe them in a different way: since
$\ZZ$ is a free group, every short exact sequence of the form $1\to H\to G\to \ZZ\to 1$
splits, and so $G$ is a semidirect product of $H$ by $\ZZ$; namely $G\simeq H\rtimes_{\alpha}
\ZZ$, for some $\alpha \in \Aut(H)$. Let us recall this well known construction in
order to fix our notation.

Given an arbitrary group $H$ and an automorphism $\alpha \in \Aut (H)$, define the
\emph{semidirect product of $H$ by $\ZZ$ determined by $\alpha$} as the group
$ H\rtimes_{\alpha} \ZZ$ with underlying set $H\times \ZZ$ and
operation given by
\begin{equation}\label{eq:semidirect product}
	(h,m) \cdot (k,n) = (h \, \alpha^{m}(k) ,m + n),
\end{equation}
for all $h,k \in H$, and $m,n \in \ZZ$. Of course, $h\mapsto (h,0)$ is a natural
embedding of $H$ in $H\rtimes_{\alpha}\ZZ$, and we then have the natural short exact
sequence
\begin{equation}
	1\to H\to H\rtimes_{\alpha}\ZZ \to \ZZ \to 1.
\end{equation}
Therefore, $H\rtimes_{\alpha}\ZZ$ belongs to the family $* \by \ZZ$. Recall that we can
have $H\rtimes_{\alpha} \ZZ \isom K\rtimes_{\beta} \ZZ$, with $H=K$ but $\alpha\neq
\beta$; or even with $H\not\simeq K$. We discuss this phenomena in
Section~\ref{sec:isom} (see~\cref{lem:suf cond for isomorphic Z-extensions} and
\cref{cor:nonisomorphic base groups}, respectively).

In the opposite direction, assume that $G$ is in the family $* \by \ZZ$. Choose a homomorphism
onto $\ZZ$, say $\rho \colon G \twoheadrightarrow \ZZ$, and consider the short exact
sequence given by
\begin{equation*}
	1\to H\to G\stackrel{\rho}{\twoheadrightarrow} \ZZ\to 1,
\end{equation*}
where $H=\ker\rho \normaleq G$. Choose and denote by $t$ a preimage in $G$ of any of the
two generators of $\ZZ$ (note that choosing such $t$ is equivalent to choosing a split
homomorphism for $\rho$). Now consider the conjugation by $t$ in $G$, say $\gamma_t \colon
G\to G$, $g\mapsto tgt^{-1}$, and denote by $\alpha \in \Aut (H)$ its restriction to
$H$ (note that $\gamma_t$ is an inner automorphism of $G$, but $\alpha$ may very well
not be inner as an automorphism of $H$). By construction, we have
\begin{equation}\label{eq:semidirect jumping}
	th =\alpha(h)t,
\end{equation}
for every $h\in H$. At this point, it is clear that every element from $G$ can be
written in a unique way as $ht^k$, for some $h\in H$ and $k\in \ZZ$. And---from
\eqref{eq:semidirect jumping}---the operation in $G$ can be easily understood by
thinking that $t$ (respectively, $t^{-1}$) jumps to the right of elements from $H$ at
the price of applying $\alpha$ (respectively, $\alpha^{-1}$):
\begin{equation}\label{prod}
	h t^{m}\cdot k t^{n} = h \, \alpha^{m}(k) \, t^{m+n}.
\end{equation}
This is, precisely, the multiplicative version of~\eqref{eq:semidirect product}. Hence,
$G\isom H\rtimes_{\alpha} \ZZ$, the semidirect product of $H$ by $\ZZ$ determined by
$\alpha$.

From this discussion it follows easily that, for any presentation of $H$, say
$H=\pres{X}{R}$, and any $\alpha\in \Aut(H)$, the semidirect product $G =
H\rtimes_{\alpha}\ZZ$ admits a presentation of the form
\begin{equation}\label{eq:presentation of a Z-extension}
	\Pres{X,t}{R \, , \, txt^{-1}=\alpha(x) \ (x\in X)}.
\end{equation}
Note that~\eqref{eq:presentation of a Z-extension} is a finite presentation if and only
if the initial presentation for $H$ was finite. So, a group $G$ admits a finite
presentation of type~\eqref{eq:presentation of a Z-extension} if and only if $G$ is
$\fp \by \ZZ$. This provides the notion of standard presentation in this context.

\begin{defn} \label{def:standard presentation}
	A \defin{standard presentation} for a $\fp \by \ZZ$ group $G$ is a \emph{finite}
	presentation of the form \eqref{eq:presentation of a Z-extension}.
\end{defn}

The previous discussion provides the following alternative descriptions for the family
of finitely presented $\ZZ$-extensions. For any group $H$, we have \begin{equation*}
\fpa{H \by \ZZ} =\{ H\rtimes_{\alpha}\ZZ \text{ f.p.}\mid \alpha \in \Aut(H)\},
\end{equation*}
and then,
\begin{align*}\label{eq:descriptions of *-by-Z}
	\notag \fpa{* \by \ZZ} \,
	  & =\,  \{\, G\ \fp\,\mid \betti{G}\geqslant 1 \, \}                                          \\
	  & =\, \{\, H\rtimes_{\alpha}\ZZ \ \fp\,\mid H\text{ group, and } \alpha \in \Aut(H) \,\} \,.
\end{align*}
\begin{rem}
	Note that
	we have made no assumptions on the base
	group $H$. Imposing natural conditions on it, we get the inclusions
	\begin{equation} \label{eq: sequence of inclusions}
		\fp \by \ZZ \, \subseteq \fpa{\fg \by \ZZ} \, \subseteq \, \fpa{*\by\ZZ} \, ,
	\end{equation}
	which will be seen throughout the paper to be both strict.
\end{rem}

The  strictness of the second inclusion in \eqref{eq: sequence of inclusions} is a direct consequence of \cref{cor:K*Z is fg-by-Z iff K=1}, while the strictness of the first one is proved below (we thank Conchita Martínez for pointing out the candidate group \eqref{eq:candidate group} in the subsequent proof).

\begin{prop} \label{prop: fp-by-Z < [fg-by-Z]fp}
	The inclusion $\fp \by \ZZ \, \subseteq \fpa{\fg \by \ZZ}$ is strict. That is, there exist finitely presented $\ZZ$-extensions of finitely generated groups, which are not $\ZZ$-extensions of any finitely presented group.
\end{prop}

\begin{proof}

	Let $p<q<r$ be three different prime numbers, and consider the additive group $A$ of the ring $\ZZ[\frac{1}{p}, \frac{1}{q}, \frac{1}{r}]$, which is well known to be generated by 
	$X = \{1/p^n\}_{n\in \NN} \cup \{1/q^n\}_{n\in \NN} \cup \{1/r^n\}_{n\in \NN}$,
	but not finitely generated
	(for any given finite set of elements in $A$, let $k$ be the biggest $p$-exponent in the denominators, and it is easy to see that $1/p^{k+1}\in A$ is not in the subgroup generated by them). Finally consider the following two commuting automorphisms $\alpha,\beta\colon A \to A$ given by $\alpha \colon a\mapsto \tfrac{p}{r}a $, and $\beta \colon a\mapsto \tfrac{q}{r}a$.

	Our candidate $G$ is the
	(metabelian)
	semidirect product of $A$ by $\ZZ^2 = \pres{t,s}{[t,s]}$, with action $t\mapsto \alpha$, $s\mapsto \beta$, namely,
	\begin{equation}  \label{eq:candidate group}
		\begin{aligned}
			G
			  & \,=\,  A \rtimes_{\alpha,\, \beta} \ZZ^2                                            \\
			  & \,=\, \Pres{A, t, s}{ts=st,\,\, tat^{-1}=\frac{p}{r}a,\, sas^{-1}=\frac{q}{r}a }  .
		\end{aligned}
	\end{equation}
	(One has to be careful here with the notation: it is typically multiplicative for the nonabelian group $G$, but additive for the abelian group $A$, while $\alpha$ and $\beta$ are defined using products of rational numbers; beware, in particular, of the element $1\in\ZZ\subseteq A$ which is additive and, of course, nontrivial.)

	It is easy to see that $G$ is generated by $1 \in A$, and $t, s \in \ZZ^2$: indeed, conjugating 1 by all powers of $t$ and $s$ we obtain, respectively, $p^n/r^n$ and $q^n/r^n$, and then $\lambda_n p^n/r^n + \mu_n q^n/r^n =1/r^n$ for appropriate integers $\lambda_n, \mu_n$, by Bezout's identity; with the same trick and having $r^n/p^n$ and $1/r^n$, we get $1/p^nr^n$ and so, $1/p^n$; and similarly, one gets $1/q^n$. Note that, in order to obtain all of $A$, it is enough to get $1/p^n$, $1/q^n$ and $1/r^n$ for $n$ big enough; this will be used later.

	To see that $G$ is finitely presented, it is enough to use
	Theorem~A(ii) in \cite{bieri_valuations_1980}, which provides a precise condition for
	a finitely generated metabelian group to be  finitely presented. This is a result, due to Bieri and Strebel, that later lead to the development of the so called Bieri--Neumann--Strebel theory (see~ \cref{sec:BNS}).

	Note that the group $G$ is finitely generated and metabelian, having $A$ as an abelian normal subgroup with quotient~$\mathbb{Z}^2$. We know that $A$ is not finitely generated as group; however, with $\mathbb{Z}^2$ acting by conjugation, $A$ becomes a $\mathbb{Z}^2$-module, which is finitely generated by the exact same argument as in the previous paragraph. But even more: for all nontrivial valuation $v\colon \mathbb{Z}^2 \to \mathbb{R}$, $A$ is also finitely generated over at least one of the two monoids $\{(n,m)\in \mathbb{Z}^2 \mid v(n,m)\geqslant 0\}$, or $\{(n,m)\in \mathbb{Z}^2 \mid v(n,m)\leqslant 0\}$. This is because any such valuation has the form $(n,m)\mapsto \alpha n+\beta m$ for some $(0,0)\neq (\alpha,\beta)\in \mathbb{Z}^2$, and then it is routine to show that, starting with $1\in A$, conjugating only either by those $t^ns^m$ with $\alpha n+\beta m\geqslant 0$, or those with $\alpha n+\beta m\leqslant 0$, and adding, we can get all of $A$ (we leave the details to the reader). By Theorem~A(ii) from \cite{bieri_valuations_1980}, this implies that the group $G$ is finitely presented.

	Now consider the subgroup $H=\langle ts, A\rangle \leqslant G$, which is clearly normal and produces a quotient $G/H=\ZZ = \pres{z}{-}$. Since
	\begin{equation*}
		H\simeq \,  A \rtimes_{\alpha\circ\beta} \ZZ
		\, =\, \Pres{A, z}{ z a z^{-1} = \frac{pq}{r^2}a}
	\end{equation*}
	is generated by $1,z$ (by the same reason as above), we deduce that $G$ is both a $\ZZ$-extension of its finitely generated subgroup $H$, and finitely presented; i.e., $G \in \fpa{\fg \by \ZZ}$.

	It remains to see that $G \notin \fp \by \ZZ$ (\ie $G$ is not a $\ZZ$-extension of any finitely presented subgroup). We do not know whether this is true for every $p,q,r$, but we shall prove it for particular values of the parameters; concretely for $(p,q,r)=(2,3,5)$.

	It is easy to see that the derived subgroup $G'$ is contained in $A$. We shall prove that, when this inclusion is indeed an equality---for example, when $(p,q,r)=(2,3,5)$, as it is straighforward to see---then $G$ is not a $\ZZ$-extension of any finitely presented subgroup. That is, no normal subgroup $N\normaleq G$ with $G/N \simeq \ZZ$ can be finitely presented. In fact, let $N\normaleq G$ be such a subgroup. Then~$A=G'\normaleq N\normaleq G$ and, taking quotients by $A$, we obtain $1\normaleq N/A\normaleq G/A=\ZZ^2 =\langle t,s\rangle$. But $\ZZ\simeq G/N \simeq \frac{G/A}{N/A}\simeq \frac{\ZZ^2}{N/A}$. So, it must be $N/A\simeq \ZZ$.

	Now, choose~$(n,m)\neq (0,0)$ such that $N/A$ is generated by $t^ns^m A$ (we can clearly assume $n>0$); and deduce that $N\simeq A \ltimes_{\varphi} \ZZ$ with action $\varphi\colon a\mapsto \frac{p^nq^m}{r^{n+m}}a$. In particular, $N$ is finitely generated by an argument as above. Note also that the action of $\varphi$ is by multiplication by a simplified fraction, say $\lambda/\mu$, with $\lambda$ and $\mu$ both different from $\pm 1$ (if $m\geqslant 0$ it is multiplication by $\frac{p^nq^m}{r^{n+m}}$; and if $m<0$, it is multiplication by~$\frac{p^n r^{|m|}}{r^nq^{|m|}}$).

	Finally, let us apply again Theorem~A(ii) in \cite{bieri_valuations_1980}, now to the short exact sequence $1\to A\to N\to \mathbb{Z}\to 1$. The only nontrivial valuations $\mathbb{Z}\to \mathbb{R}$ are $1\mapsto \pm 1$, and it is easy to see that $A$ is not finitely generated neither as a $\mathbb{Z}^+$-module (with finitely many elements one cannot obtain $1/\lambda^n$ for $n\gg 0$), nor as a $\mathbb{Z}^-$-module (with finitely many elements one cannot obtain $1/\mu^n$ for $n\gg 0$). Therefore, $N$ is not finitely presented, and the group~$G$ is not a $\ZZ$-extension of any finitely presented subgroup, as we wanted to prove.
\end{proof}

\section{Unique $\ZZ$-extensions}\label{unique}

Recall that the family of unique $\ZZ$-extensions (\ie groups having a unique normal
subgroup with quotient $\ZZ$) which are finitely presented is denoted $\fpa{!\by\ZZ}$.

As seen in \cref{lem:Betti1 and rank of a free-abelian quotient}, the family $\fpa{!
	\by \ZZ}$ consists precisely of those groups $G$ in~$\fpa{* \by \ZZ}$ such that
$\betti{G}=1$. For finitely generated base groups $H$, \cref{prop:deranged
characterization} gives a quite simple characterization of this unicity condition in
terms of the defining automorphism $\alpha$. We first need a convenient description of
the abelianization of a $\ZZ$-extension.

\begin{lem}\label{lem:abelianization of a Z-extension}
	Let $H$ be an arbitrary group, and let $\alpha \in \Aut(H)$. Then,
	\begin{equation}  \label{eq:abelianization of a *-by-Z extension}
		(H\rtimes_{\alpha}\ZZ)\ab \, \isom \, \frac{H\ab}{\im(\alpha\ab-\id)} \, \oplus \, \ZZ \,.
	\end{equation}
	Moreover, if $H$ is finitely generated, then so is $H\rtimes_{\alpha} \ZZ$, and
	\begin{equation} \label{eq:abelianization of a fg-by-Z extension}
		(H\rtimes_{\alpha} \ZZ)\ab \isom \ZZ^{k+1} \oplus T,
	\end{equation}
	where $k$ is the rank of $\ker (\alpha_*\ab-\id)$, and $T$ is a finite abelian group.
\end{lem}

\begin{proof}
	Let $H=\pres{X}{R}$. Abelianizing $H\rtimes_{\alpha} \ZZ =\pres{X,t}{R ,\, tx_i
		t^{-1}=\alpha (x_i) \,\,\, (x_i \in X)}$, we get
	\begin{align}\label{ccc}
		(H\rtimes_{\alpha} \ZZ)\ab        & = \Pres{X,t}{
		\begin{smallmatrix}
		R,\hfill\\
		t x_i t^{-1} =\alpha (x_i) \hfill & \hfill (x_i \in X),                                                    \\
		x_i x_j =x_j x_i \hfill           & \ \hfill (x_i,x_j \in X),                                              \\
		t x_i =x_i t \hfill               & \hfill (x_i \in X)\phantom{,}
		\end{smallmatrix} \notag
		}\\[3pt]
		                                  & \simeq \, \Pres{X}{
		\begin{smallmatrix}
		R,\hfill\\
		x_i = \alpha (x_i)\hfill          & \hfill (x_i \in X),                                                    \\
		x_i x_j =x_j x_i \hfill           & \ \hfill (x_i,x_j \in X)
		\end{smallmatrix}
		}
		\times \pres{t}{-}
		\\[3pt] \notag
		                                  & \simeq\, \frac{H\ab}{\im ( \alpha\ab - \id)} \, \oplus \, \pres{t}{-} \, .
	\end{align}
	For the second part, suppose that $H$ is finitely generated. Then so is $H\ab$ and
	thus, using \cref{lem:Betti1 properties}, we have
	\begin{equation} \label{eq: Betti fg-by-Z}
		\begin{aligned}
			\betti{H\rtimes_{\alpha}\ZZ}
			  & = \Betti{\frac{H\ab}{\im ( \alpha\ab - \id)}}+1       \\
			  & = \betti{H\ab} - \betti{\im ( \alpha\ab - \id)} +1    \\
			  & = \betti{H\ab_*}-  \betti{\im (\alpha_*\ab - \id)} +1 \\
			  & = \betti{\ker (\alpha_*\ab - \id)} +1                 \\
			  & = \rk \,(\ker (\alpha_*\ab - \id)) +1,
		\end{aligned}
	\end{equation}
	which is what we wanted to prove.
\end{proof}

This last result, combined with \cref{lem:Betti1 and rank of a free-abelian quotient},
provides a computable characterization for automorphisms defining $!\fg\by\ZZ$ groups.

\begin{prop}\label{prop:deranged characterization}
	Let $H$ be an arbitrary group, and $\alpha \in \Aut(H)$ such that the semidirect
	product $ H\rtimes_{\alpha}\ZZ$ is finitely generated. Then, the following conditions
	are equivalent:
	\begin{enumerate}[(a)]
		\item \label{item:!-by-Z}$H\rtimes_{\alpha}\ZZ$ is $! \by \ZZ$;
		\item \label{item:betti=1}$\betti{H\rtimes_{\alpha}\ZZ }=1$;
		\item \label{item:H=preabelianization of torsion} $H=\ker(\piabf)$ (\ie $H$ is the
		      full preimage of the torsion subgroup of $(H\rtimes_{\alpha}\ZZ)\ab$ under
		      $\piab$);
		\item \label{item:characteristic}$H$ is a fully characteristic subgroup of
		      $H\rtimes_{\alpha}\ZZ$.
	\end{enumerate}
	Moreover, if $H$ is finitely generated, then the following additional condition is also
	equivalent:
	\begin{enumerate}[(a),resume]
		\item \label{item:deranged} $\alpha\abf$ has no nontrivial fixed points
		      (equivalently, $1$ is not an eigenvalue of $\alpha\abf$, $\ker(\alpha\abf -
		      \id)=0$, or $\det(\alpha\abf - \id )\neq 0$).
	\end{enumerate}
\end{prop}

\begin{proof}
	$[\ref{item:!-by-Z}\Leftrightarrow \ref{item:betti=1}$]. This is precisely the content
	of \cref{lem:Betti1 and rank of a free-abelian quotient}(ii), for $k=1$.

	$[\ref{item:betti=1} \Leftrightarrow \ref{item:H=preabelianization of torsion}]$. This follows immediately from~\eqref{ccc}.

	$[\ref{item:H=preabelianization of torsion} \Rightarrow \ref{item:characteristic}]$.
	This is clear, since the torsion subgroup of an abelian group is fully characteristic,
	and so is its full preimage.

	$[\ref{item:characteristic} \Rightarrow \ref{item:betti=1}]$. By contradiction, suppose
	that $H\rtimes_{\alpha} \pres{t}{-}$ has Betti number at least~$2$. Then, there exists
	an epimorphism $\rho \colon H\rtimes_{\alpha} \gen{t} \onto \ZZ^2$, and an element
	$h\in H$ such that $\rho(h)\neq 0$. Take a primitive element $v\in \ZZ^2$ such that
	$\rho(h)=\lambda v$ for some $\lambda \in \ZZ$, $\lambda \neq 0$. The subgroup $\gen{v}\simeq \ZZ$ is a direct summand of $\ZZ^2$, and the composition
\begin{equation*}
\arraycolsep=2.5pt
\begin{array}{ccccccl}
H\rtimes_{\alpha} \gen{t} & \overset{\rho}{\onto} & \ZZ^2 & \onto & \gen{v} & \into &
H\rtimes_{\alpha} \gen{t} \\ &&&& v\ & \mapsto & t
\end{array}
\end{equation*}
	provides an endomorphism of $H\rtimes_{\alpha}\ZZ$ mapping $h\in H$ to $t^{\lambda}$, a contradiction with condition~\ref{item:characteristic}.

	Finally, the equivalence $[\ref{item:betti=1}\Leftrightarrow \ref{item:deranged}]$ is immediate from the equality \eqref{eq: Betti fg-by-Z}.
\end{proof}

\begin{defn}\label{def:deranged}
	We say that an automorphism $\alpha \in \Aut(H)$ is \defin{deranged} if one of (and thus all) the conditions $\ref{item:!-by-Z}-\ref{item:characteristic}$ in \cref{prop:deranged characterization} hold. Note that when the group $H$ is finitely generated, condition \ref{item:deranged} provides further another equivalent definition
	of derangedness,
	this time expressed in terms of the automorphism. Note that in this last case, given $\alpha \in \Aut(H)$ by images of generators, it is easy to check algorithmically whether $\alpha$ is deranged or not.
\end{defn}

In particular, every automorphism of a group $H$ with $\betti{H}=0$, \ie with
$|H\ab|<\infty$, is (trivially) deranged. Note also that derangedness is, in fact, a
property of outer automorphisms. The sets of deranged automorphisms and deranged outer
automorphisms of a group $H$ will be denoted, respectively, $\Aut_{\der}(H)$ and
$\Out_{\der}(H)$.

Consequently, for any finitely generated group $H$ we have
\begin{equation*}
	\fpa{!H \by \ZZ} =\{ \,H\rtimes_{\alpha}\ZZ \ \fp \mid \alpha \in \Aut_{\der}(H) \,\},
\end{equation*}
and then,
\begin{align} \label{eq:descriptions of !fg-by-Z}
	\notag \fpa{!\fg \by \ZZ} \,
	   & = \, \{ \, H\rtimes_{\alpha}\ZZ \ \fp \mid H\ \fg \text{ and }\alpha \in \Aut_{\der}(H) \,\} \\
	\, & \subseteq \, \fpa{! \by \ZZ} \ = \ \{\, G \ \fp \mid \betti{G}=1 \} \, .
\end{align}
Note that in \eqref{eq:descriptions of !fg-by-Z} we wrote inclusion and not an equality
because, in principle, it could happen that a finitely presented $! \by \ZZ$ group has
his unique normal subgroup with quotient isomorphic to $\ZZ$  being not finitely
generated. In the next section we shall construct such a group (see~\cref{cor:K*Z is
	fg-by-Z iff K=1}) showing that this inclusion is strict.

\section{Undecidability results}\label{sec:undecidability}

Observe that if $H$ is finitely generated or finitely presented, then so is
$H\rtimes_{\alpha}\ZZ$ for every $\alpha \in \Aut(H)$, \ie
\begin{align*}
	\fg \by \ZZ & \, \subseteq \, \fga{* \by \ZZ}, 
	\\
	\fp \by \ZZ & \, \subseteq \, \fpa{* \by \ZZ} \, .
\end{align*}

However, it is less obvious that the converse is not true in general: a semidirect
product $H\rtimes_{\alpha}\ZZ$ can be finitely presented, with $H$ not being finitely
generated. Or, as was hinted few lines above, even worse: there do exist \emph{finitely
	presented} $\ZZ$-extensions which are \emph{not} $\ZZ$-extensions of any finitely
generated group. In other words, the following inclusion is strict:
\begin{equation*}
	\fpa{\fg \by \ZZ} \,\subset \, \fpa{* \by \ZZ} \, .
\end{equation*}
Indeed, this can happen even for unique $\ZZ$-extensions. This fact follows easily from the next lemma, showing that any free product $K\freeprod \ZZ$ has the form of a certain semidirect product.

\begin{lem}\label{lem:*K semidirect Z = K*Z}
	Let $K=\pres{X}{R}$ be an arbitrary group with generators $X = \{x_j\}_{j\in J}$, and
	consider the free product
	\begin{equation*}
		\Freeprod_{i\in \ZZ} K=\Pres{X^{(i)} \, (i\in \ZZ)}{R^{(i)} \, (i\in \ZZ)} \, ,
	\end{equation*}
	where $\Pres{X^{(i)}}{R^{(i)}}$ $(i \in \ZZ)$ are disjoint copies of the original
	presentation for~$K$. Then,
	\begin{equation}\label{eq:*K semidirect Z = K*Z}
		\left(\Freeprod_{i\in \ZZ} K \right) \rtimes_\tau \ZZ \ \isom  \ K \freeprod \ZZ\,
		,\end{equation}
		where $\tau$ is the automorphism of $\Freeprod_{i\in \ZZ} \, K$ defined by
		\begin{equation} \label{eq:translation automorphism}
			\tau \colon x_j^{(i)} \mapsto x_j^{(i+1)} \ \left(\forall i \in \ZZ \, , \, \forall x_j^{(i)} \in X^{(i)}\right) .
		\end{equation}
		\end{lem}

		\begin{proof}
			Naming $t$ the generator of $\ZZ$, we have
			\begin{align}
				\!\!\left(\Freeprod_{i\in \ZZ} \, K \right) \rtimes_\tau \ZZ
				  & \,=\,  \Pres{ t, X^{(i)} \ (i\in \ZZ)} {\!\!
				\begin{array}{l}
				R^{(i)} \ (i\in \ZZ), \notag\\
				t \, x_j^{(i)} \, t^{-1} =x_j^{(i+1)} \ \left( i\in \ZZ,\, x_j^{(i)}\! \in X^{(i)}\right) \!\!
				\end{array}
				} \\[3pt]
				  & \,\isom\,
				\Pres{ t, X^{(i)} \ (i\in \ZZ)} {\!\!
				\begin{array}{l}
				R^{(0)}, \\
				\label{eq:infinite free product rtimes Z}
				t \, x_j^{(i)} \, t^{-1} =x_j^{(i+1)} \ \left( i\in \ZZ,\, x_j^{(i)}\! \in  X^{(i)}\right) \!\!
				\end{array}
				} \\[3pt]
				\label{eq:free product}
				  & \,\isom\,
				\Pres{ t, X^{(0)} } {R^{(0)}} \,=\, K * \ZZ \,.
			\end{align}
			To see the last isomorphism, consider the maps from~\eqref{eq:free product}
			to~\eqref{eq:infinite free product rtimes Z} given by
			\begin{align}
				t         & \mapsto t, \notag                                  \\
				x_j^{(0)} & \mapsto x_j^{(0)} \ (x_j^{(0)}\in X^{(0)}), \notag
			\end{align}
			and from~\eqref{eq:infinite free product rtimes Z} to~\eqref{eq:free product} given by
			\begin{align}
				t         & \mapsto t, \notag                                                         \\
				x_j^{(i)} & \mapsto t^{i}x_j^{(0)}t^{-i} \ (i\in \ZZ, \ x_j^{(i)}\in X^{(i)}). \notag
			\end{align}
			It is straightforward to see that they are both well-defined homomorphisms, and one inverse to the other.
		\end{proof}

		\begin{cor}\label{cor:K*Z is fg-by-Z iff K=1}
			If $K$ is a group with finite abelianization  (\ie $\betti{K}=0$), then the free
			product $K \freeprod \ZZ$ is a unique $\ZZ$-extension, and the following conditions
			are equivalent:
			\begin{enumerate}[(a)]
				\item $K\freeprod \ZZ \text{ is }\fg \by \ZZ$;
				\item $K\freeprod \ZZ \text{ is }\fp \by \ZZ$;
				\item $K\freeprod \ZZ \text{ is }\propsty{abelian} \by \ZZ$;
				\item $K\freeprod \ZZ \text{ is }\propsty{finite} \by \ZZ$;
				\item $K\freeprod \ZZ \text{ is }\propsty{free} \by \ZZ$;
				\item $K=\trivial$.
			\end{enumerate}
			In particular, $!\fg \by \ZZ$ is a strict subfamily of $\fpa{! \by \ZZ}$ (and so,
			$\fg \by \ZZ$ is a strict subfamily of $\fpa{* \by \ZZ}$).
		\end{cor}

		\begin{proof}
			Note that the abelianization of $K\freeprod \ZZ$ is $(K\freeprod \ZZ)\ab =K\ab \oplus
			\ZZ$, where $|K\ab|<\infty$ by hypothesis; therefore, $\betti{K \freeprod \ZZ} =1$.
			Thus, from~\cref{lem:Betti1 and rank of a free-abelian quotient}(ii), $K\freeprod \ZZ$
			is a unique $\ZZ$-extension, \ie it contains a unique normal subgroup with quotient
			$\ZZ$. By~\cref{lem:*K semidirect Z = K*Z}, this unique normal subgroup is
			isomorphic to $\Freeprod_{z\in \ZZ} K$, which is finitely generated (resp., finitely
			presented, abelian, finite, free) if and only if $K$ is trivial (the free case being true because $\betti{K}=0$).

			Taking $K$ to be a nontrivial finitely presented group with finite abelianization, we obtain that $K\freeprod \ZZ$ belongs to $\fpa{! \by \ZZ}$ but not to !$\fg \by \ZZ$.
		\end{proof}

		Next, inspired by a trick initially suggested by Maurice Chiodo, we will prove a stronger result. Not only the family !$\fg \by \ZZ$ is a strict subfamily of $\fpa{! \by \ZZ}$, but the membership problem between these two families is undecidable: it is impossible
		to decide algorithmically whether a given finitely presented unique $\ZZ$-extension is
		$\fg \by \ZZ$ or not, \ie whether its unique base group is finitely generated or not.
		To see this, we use a classic undecidability result: there is no algorithm which, on input a finite presentation, decides whether the presented group is trivial or not (see, for example,~\cite{miller_iii_decision_1992}).

		\begin{thm}\label{thm:fg-by-Z is undecidable}
			For every group property $\propi \in \{ \fg, \fp, \propsty{abelian}, \propsty{finite}, \propsty{free}\}$, the membership problem for \,$\propi \by \ZZ$\, within\, $\fpa{! \by \ZZ}$ is undecidable.

			In other words, there exists no algorithm which, on input a finite presentation of a
			group with Betti number $1$, decides whether it presents a $\fg \by \ZZ$ (resp., $\fp
			\by \ZZ$, $\propsty{abelian}\by \ZZ$, $\propsty{finite}\by \ZZ$, $\propsty{free}\by \ZZ$) group or not.
		\end{thm}

		\begin{proof}
			We will proceed by contradiction. Assume the existence of an algorithm, say $\mathfrak{A}$, such that, given as input a finite presentation of a group with Betti
			number $1$, outputs \YES\ if it presents a $\propi \by \ZZ$ group, and \NO\ otherwise.

			Now, consider the following algorithm $\mathfrak{B}$ to check triviality: on input an arbitrary finite presentation $K=\pres{X}{R}$:
			\begin{enumerate}[(i)]
				\item abelianize $K$ and, using the Classification Theorem for finitely generated
				      abelian groups, check whether $K\ab$ is trivial or not; if not, answer \NO;
				      otherwise $K$ is a perfect group and so, the new group $K\freeprod \ZZ$ has Betti
				      number 1;
				\item apply $\mathfrak{A}$ to the presentation $\pres{X,t}{R}$, to decide whether
				      $K\freeprod \ZZ$ is a $\propi \by \ZZ$ group or not.
			\end{enumerate}
			According to~\cref{cor:K*Z is fg-by-Z iff K=1}, the output to step (ii) is \YES\ if and only if $K$ is trivial. Hence, $\mathfrak{B}$ is deciding whether the given presentation
			$\pres{X}{R}$ presents the trivial group or not. This contradicts Adian--Rabin's Theorem
			on the undecidability of the triviality problem.
		\end{proof}

		Of course, if the membership problem is not decidable within some family $\HH$, it is
		also undecidable within any superfamily of $\HH$. So, we immediately get the following
		consequence.

		\begin{cor} \label{cor:membership for fg-by-Z and fp-by-Z is undecidable}
			The membership problems for the families \,$\fp \by \ZZ$ and $\fg \by \ZZ$ are undecidable. \qed
		\end{cor}

		As stated in the introduction, this is exactly the same as saying that the families
		(of finite presentations) $\fp \by \ZZ$ and $\fg \by \ZZ$ are not recursive. Note that none of these families is neither Markov nor co-Markov, and thus the two undecidability results in \cref{cor:membership for fg-by-Z and fp-by-Z is undecidable} are not contained in the classic ones due to Adian--Rabin. Indeed,
		any finitely presented group is a subgroup of some $\fp \by \ZZ$ (and so, of some $\fg \by \ZZ$) group; therefore the families $\fp \by \ZZ$ and $\fg \by \ZZ$ are not Markov. On the other hand, every $\fp$ group embeds in some $2$-generated simple  group (see \cite[Corollary 3.10]{miller_iii_decision_1992}); since
		\begin{equation*}
			\propsty{simple} \Imp \propsty{perfect} \Imp  \neg (\fg\by\ZZ) \Imp \neg(\fp\by\ZZ) \, ,
		\end{equation*}
		the families $\fp \by \ZZ$ and $\fg \by \ZZ$ are not co-Markov either.

		\section{Implications for the BNS invariant}\label{sec:BNS}

		Since the early 1980's, in a series of papers by R.~Bieri, W.~Neumann, and R.~Strebel (see~\cite{bieri_valuations_1980,bieri_geometric_1987}), several gradually more general invariants---called Sigma (or BNS) Invariants---have been introduced to deal with finiteness conditions for presentations of groups. Concretely in~\cite{bieri_geometric_1987}, they present an invariant that characterizes those normal subgroups of a finitely generated group $G$ that are finitely generated and contain the commutator $[G,G]$ of $G$. Over the years, this theory has been reformulated in more geometric terms (for a modern version see the survey~\cite{strebel_notes_2012}). Below, we recall this construction and characterization, and discuss some implications of our undecidability results from \cref{sec:undecidability}.

		For a finitely generated group $G$, consider the real vector space $\Hom(G,\RR)$ of all
		homomorphisms $\chi\colon G \to \RR$ (from $G$ to the additive group of the field of
		real numbers), which we call \defin{characters} of $G$. Note that, since $\RR$ is
		abelian and torsion-free, any character $\chi$ must factor
		through~$\piabf$~(abelianizing and then killing the torsion), \ie
		\begin{equation*}
			\chi \colon G\stackrel{\piab}{\twoheadrightarrow} G\ab \twoheadrightarrow G\abf \rightarrow \RR \,.
		\end{equation*}
		Thus, $\Hom(G,\RR)=\Hom(\ZZ^r,\RR)=\RR^r$, where $r=\betti{G}$. We will consider the
		set of nontrivial characters modulo the equivalence relation given by positive scaling:
		\begin{equation} \label{eq:equivalence modulo positive product}
			\chi_1 \sim \chi_2 \Biimp \exists \lambda>0 \text{ s.t. } \chi_2 = \lambda \chi_1\,.
		\end{equation}
		They form the so-called \defin{character sphere} of $G$, denoted
		$\sph{G}=\Hom(G,\RR)^*/\sim$ which, equipped with the quotient topology, is
		homeomorphic to the unit Euclidean sphere of dimension ${r-1}$ (through the natural
		identification of each ray emanating from the origin with its unique point of norm 1).

		For example, if $G$ is not $* \by \ZZ$ (\ie if $\betti{G}=0)$, then $\Hom(G,\RR) =
		\{0\}$ and the character sphere is empty (so, for this class of groups the BNS theory
		will be vacuous). More interestingly, if $G$ is $! \by \ZZ$ (\ie $\betti{G}=1$), then
		the character sphere of $G$ is a set of just two points, namely $\sph{G} = \{+1,-1\}$.
		Similarly, if
		$\betti{G}=2,3,\ldots$, then $\sph{G}$ is the unit circle in $\RR^2$, the unit sphere
		in $\RR^3$, and so on.


		For any given (equivalence class of a) nontrivial character $\chi$, consider now the
		following submonoid of $G$, called the \defin{positive cone} of $\chi$:
		\begin{equation}\label{eq:positive cone}
			G_\chi = \{ g \in G \mid \chi(g) \geq 0\} = \chi^{-1}([0,+\infty))\,,
		\end{equation}
		to be thought of as the full subgraph of the Cayley graph $\Gamma(G,X)$ determined by the vertices in $G_{\chi}$ (once a set of generators $X$ is fixed). The Sigma invariant $\bns{G}$ can then be defined as follows (we note that this is not the original definition given in~\cite{bieri_geometric_1987}, but a more geometrically appealing one, which was not noticed to be equivalent until several years later, see~\cite[Theorem 3.19]{meigniez_bouts_1990}).

		\begin{defn} \label{def:BNS invariant}
			Let $G=\langle X\rangle$ be a finitely generated group, and $\Gamma(G,X)$ its Cayley
			graph. Then the set
			\begin{equation}
				\bns{G} = \left\{\, [\chi] \in \sph{G} \mid G_\chi \text{ is connected} \right\} \subseteq \sph{G}
			\end{equation}
			does not depend on the choice of the finite generating set $X$ (see~\cite{strebel_notes_2012}), and is called the \defin{(first) Sigma}---or \defin{BNS}---\defin{invariant} of $G$.
		\end{defn}

		Interestingly, this notion is quite related with commutativity. The
		extreme examples are free and free-abelian groups, for which it is easy to see that the
		BNS invariants are, respectively, the empty set and the full character sphere:
		$\bns{F_r}=\emptyset$, for $r\geqslant 2$; and $\bns{\ZZ^r}=\sph{\ZZ^r}$, for
		$r\geqslant 1$.

		The set of characters vanishing on a certain subgroup $H\leqslant G$ determine the
		following subsphere
		\begin{equation*}
			\sph{G, H} = \{[\chi] \in \sph{G} \mid \chi(H) = 0 \}\subseteq \sph{G} \, ,
		\end{equation*}
		which happens to contain interesting information about $H$ itself.

		\begin{thm}[Bieri--Neumann--Strebel, 
			\cite{bieri_geometric_1987}]\label{thm:H fg iff S(H) in Sigma} Let $H$ be a normal subgroup of a finitely generated group $G$ with $G/H$ abelian. Then, $H$ is finitely generated if and only if $\sph{G,H} \subseteq \bns{G}$. In particular, the commutator subgroup $[G,G]$ is finitely generated if and only if $\bns{G}=\sph{G}$. \qed
		\end{thm}

		Note that if $G$ is $H\by \ZZ$, then $\sph{G,H}=\{[\pi_H], -[\pi_H]\}$, where $\pi_H
		\colon G\twoheadrightarrow G/H\isom \ZZ$ is the canonical projection modulo $H$. In
		this case, \cref{thm:H fg iff S(H) in Sigma} tells us that
		\begin{equation}\label{eq:H fg iff pi,-pi in Sigma}
			H\text{ is }\fg \Biimp [\pi_H], -[\pi_H] \in \bns{G} \, .
		\end{equation}
		It follows an interesting characterization of $\fg \by \ZZ$ groups.

		\begin{prop}\label{antipodal}
			A finitely generated group $G$ is $\fg \by \ZZ$ if and only if its BNS invariant contains a pair of antipodal points; \ie
			\begin{equation}
				G\text{ is } \fg \by \ZZ \Biimp \exists [\chi] \in \sph{G} \text{ s.t. } [\chi], -[\chi] \in \bns{G} \,.
			\end{equation}
		\end{prop}

		\begin{proof}
			The implication to the right is clear from~\eqref{eq:H fg iff pi,-pi in Sigma}.

			The implication to the left is also clear after making sure that we can always choose
			such a character $\chi$ with cyclic image (\ie such that $\rk_{\ZZ}\chi(G)=1$). To see
			this, we observe that, given a nontrivial character $\chi\colon G\to \RR$, one has
			$\rk_{\ZZ}\chi(G)=1$ if and only if there exists $\lambda >0$ such that $\lambda\chi$
			has integral image, $\lambda \chi\colon G\twoheadrightarrow \ZZ\subseteq \RR$. In other
			words, rank-one characters correspond, precisely, to those points in the sphere $\sph{G}$
			which are projections of integral (or rational) points from $\RR^r\setminus \{ 0\}$.
			Thus, rank-one characters form a dense subset of $\sph{G}$. This, together with the
			fact that $\bns{G}$ is an open subset of $\sph{G}$ (see~\cite[Theorem
			A3.3]{strebel_notes_2012}) allows us to deduce, from the hypothesis, the existence of a
			pair of antipodal points of rank one.
		\end{proof}

		As a corollary, and using~\cref{thm:fg-by-Z is undecidable}, we obtain the main result
		in this section: the BNS invariant is not uniformly decidable (even for groups with Betti number $1$).

		\begin{thm} \label{thm:BNS is not decidable for Betti 1}
			There is no algorithm such that, given a finite presentation of a group~$G$ (with Betti number $1$),  and a character $[\chi]\in \sph{G}$, decides whether $[\chi]$ belongs to $\bns{G}$ or not.
		\end{thm}

		\begin{proof}
			Given a finite presentation of a $! \by \ZZ$ group $G$ (so $\betti{G}=1$, and $\sph{G}$
			has two points), we can abelianize and construct the unique two characters $\pm \pi
			\colon G\twoheadrightarrow \ZZ$. Assuming the existence of an algorithm like the one in the statement, we could
			algorithmically decide whether $\pi, -\pi$ both belong to $\bns{G}$ or not, \ie
			according to~\cref{antipodal}, whether $G$ is $\fg \by \ZZ$ or not. This
			contradicts~\cref{thm:fg-by-Z is undecidable}.
		\end{proof}
		%

		We note that, in the case of a one-relator group $G = \pres{a,b}{r}$, K.~Brown provided an interesting algorithm for deciding whether a given character $\chi \colon G \to \RR$ belongs to~$\bns{G}$ or not, by looking at the sequence of $\chi$-images of the prefixes of the relation~$r$~(assumed to be in cyclically reduced form); see~\cite{brown_trees_1987}.
		Later, N.~Dunfield, J.~Button and D.~Thurston found applications of this result to $3$-manifold theory; see \cite{dunfield_alexander_2001,button_fibred_2005,dunfield_random_2006}.

		\section{Recursive enumerability of presentations} \label{sec:searching standard presentations}

		A \emph{standard} presentation of a given $\fp \by \ZZ$ group $G$ has been defined as a
		\emph{finite} presentation of the form
		\begin{equation*}
			\Pres{X,t}{R(X)\,,\, txt^{-1}=\alpha(x) \ (x\in X)},
		\end{equation*}
		where $\alpha$ is an automorphism of $\pres{X}{R}$. It is natural to ask for an
		algorithm to compute one---or all---standard presentations for such a group $G$,
		since this algorithm will provide explicit computable ways to think $G$ as a
		semidirect product (\ie an explicit base group $H = \pres{X}{R}$, and an explicit automorphism
		$\alpha$, such that $G \isom H \rtimes_{\alpha} \ZZ$). 

		We have
		seen that membership for $\fp \by \ZZ$ is undecidable
		(\cref{cor:membership for fg-by-Z and fp-by-Z is undecidable}). However, given a finite
		presentation for a $\fp \by \ZZ$ group $G$, we can use Tietze transformations to obtain
		a recursive enumeration of all the finite presentations for $G$. In the following
		proposition we provide a (brute force) filtering process which extracts from it a
		recursive enumeration of all the standard ones.

		\begin{prop}\label{prop:enumerate standard presentations of fg-by-Z}
			Given a finite presentation of a $\fp \by \ZZ$ group $G$, the set of standard
			presentations for $G$ is recursively enumerable.
		\end{prop}

		\begin{proof}
			Let $P$ be the finite presentation given (of a $\fp \by \ZZ$ group $G$). We will start enumerating all finite presentations of $G$ by successively applying to $P$ chains of elementary Tietze transformations in all possible ways. This process is recursive and eventually visits all finite presentations for $G$ (all standard presentations among them).

			Now, it will be enough to construct a recognizing subprocess $\mathfrak{S}$ which,
			applied to any finite presentation $P'$ for $G$, if $P'$ is in standard form it halts and returns $P'$, and if not it halts returning ``\NO, $P'$ is not standard", or works forever. Having $\mathfrak{S}$, we can keep following the enumeration of all finite presentations $P'$ for $G$ via Tieze transformations and, for each one, start and run in parallel the recognizing process $\mathfrak{S}$ for it; we maintain all of them running in parallel (some of them possibly forever), and at the same time we keep opening new ones, simultaneously aware of the possible halts (each one killing one of the parallel processes and possibly outputting a genuine standard presentation for $G$).

			So, we are reduced to design such a recognizing process $\mathfrak{S}$. For a given finite presentation $P'$ of $G$, let us perform the following steps:
			\begin{enumerate}[(i)]
				\item Check whether $P'$ matches the scheme
				      \begin{equation}\label{eq:candidate presentation}
				      	\Pres{X,t}{R \, , \, t x_i t^{-1} = w_i \,\,\, (x_i \in X)},
				      \end{equation}
				      where $X=\{x_1,\ \ldots ,x_n\}$ and $R=\{ r_1, \ldots ,r_m \}$ are finite, and the
				      $w_i$'s and $r_j$'s are all (reduced) words on $X$. If $P$ does not match this
				      scheme, then halt and answer ``\NO, $P'$ is not standard"; otherwise
				      go to the next step.
				\item With $P'$ being of the form~\eqref{eq:candidate presentation}, consider the
				      group $H=\pres{X}{R}=F(X)/\normalcl{R}$ and let us try to check whether the map
				      $x_i \mapsto w_i$ extends to a well-defined homomorphism $\alpha\colon H\to H$.
				      For this, we must check whether $\alpha(r_j)=1$ in $H$ or not (but caution! we
				      cannot assume in general a solution to the word problem for $H$). Enumerate and
				      reduce the elements in $\normalcl{R}$ and check whether, for every relator
				      $r_j(x_1,\ldots,x_n)\in R$, the word $r_j(w_1,\ldots,w_n)$ appears in the
				      enumeration. If this happens for all $j=1,\ldots ,m$, then go to the next step
				      (with $P'$ being of the form
				      \begin{equation}\label{output-step2}
				      	\Pres{X,t}{R \, , \, t x_i t^{-1} = \alpha(x_i) \,\,\, (x_i \in X)},
				      \end{equation}
				      where $\alpha \in \End(F(X)/\!\normalcl{R})$ ).
				\item With $P'$ being of the form~\eqref{output-step2}, let us try to check now
				      whether $\alpha$ is bijective, looking by brute force for its eventual inverse:
				      enumerate all possible $n$-tuples $(v_1, \ldots ,v_n)$ of reduced words on $X$
				      and for each one, check simultaneously whether $r_j(v_1, \ldots ,v_n)=1$ in $H$
				      for all $j=1,\ldots ,m$ (\ie whether $\beta\colon H\to H$, $x_i\mapsto v_i$ is a
				      well-defined endomorphism of $H$) and whether $v_i(w_1,\ldots ,w_n)=1$ and
				      $w_i(v_1,\ldots ,v_n)=1$ in $H$, for all $i=1,\ldots ,n$ (\ie whether
				      $\alpha\beta=\beta\alpha=\id$ and so $\alpha\in \Aut(H)$). We do this in a
				      similar way as in the previous step: enumerate the normal closure $\normalcl{R}$
				      (an infinite process) and wait until all the mentioned words appear in the
				      enumeration. When this happens (if so), halt the process and output $P'$ as a
				      standard presentation for $G$.
			\end{enumerate}

			For any given $P'$, step~(i) finishes in finite time and either rejects $P'$, or
			recognizes that $P'$ is of the form~\eqref{eq:candidate presentation} and sends the
			control to step~(ii). Now step~(ii) either works forever, or it halts recognizing $P'$
			of the form~\eqref{output-step2} and sending the control to step~(iii) (note that, by
			construction, it is guaranteed that if $P'$ is really in standard form then $\alpha$ is
			a well-defined endomorphism of $H$ and step~(ii) will eventually halt in finite time).
			Finally, the same happens in step~(iii): it either works forever, or it halts
			recognizing that $P'$ is in standard form (again by construction, it is guaranteed that if $P'$ is really in standard form, then $\alpha$ is bijective and step~(iii) will
			eventually catch its inverse and halt in finite time).

			Process $\mathfrak{S}$ is built, and this concludes the proof.
		\end{proof}

		We remark that we can apply the previous algorithm to an arbitrary finite presentation $P$ of a (arbitrary) group $G$: if $G$ is a $\fp\by\ZZ$ group the process will enumerate all its standard presentations, while if $G$ is not $\fp\by\ZZ$ the process will work forever outputting nothing. So, we can successively apply---in parallel---the previous
		algorithm to any enumerable family~$\HH$ of presentations to obtain an enumeration of all standard $\fp\by\ZZ$ presentations within $\HH$. Taking~$\HH = \GG_{\fp}$, we get an enumeration of all standard $\fp\by\ZZ$ presentations.

		\begin{cor}
			The set of standard presentations of $\fp\by\ZZ$ groups is recursively enumerable. \qed
		\end{cor}

		Applying all possible Tietze transformations to every standard presentation outputted by this procedure, we obtain an enumeration of all finite presentations of $\fp\by\ZZ$ groups. This enriches \cref{cor:membership for fg-by-Z and fp-by-Z is undecidable} in the following way.

		\begin{cor}
			The set of finite presentations of $\fp\by\ZZ$ groups is recursively enumerable but not recursive. \qed
		\end{cor}

		\section{On the isomorphism problem for unique $\ZZ$-extensions}\label{sec:isom}

		Let us consider now problems of the first kind mentioned in \cref{recog}: isomorphism problems within families of the form $\fpa{\propi \by \ZZ}$.

		\medskip

		To begin with, we combine \cref{lem:*K semidirect Z = K*Z} with the following one to see that a $\ZZ$-extension can have non isomorphic base groups. The proof is just a direct writing of the corresponding presentations.

		\begin{lem}\label{lem:rtimes times = times rtimes}
			Let $H$ be an arbitrary group, and $\phi \in \Aut(H)$. Then,
			\begin{equation} \label{eq:rtimes times = times rtimes}
				(H\rtimes_{\phi} \ZZ)\times \ZZ  \, \isom  \, (H\times \ZZ )\rtimes_{\Phi} \ZZ \, ,
			\end{equation}
			where $\Phi
			\in \Aut(H\times \ZZ)$ is defined by
			$(h,t)\mapsto (\phi(h), t)$. \qed
		\end{lem}


		\begin{cor}\label{cor:nonisomorphic base groups}
			Isomorphic $\ZZ$-extensions can have nonisomorphic base groups, even of different type. More precisely, there exist a finitely presented group $H$, a non finitely generated group $H'$, and automorphisms $\alpha \in \Aut(H)$ and $\beta\in \Aut(H')$, such that $H \rtimes_{\alpha}\ZZ \simeq H'\rtimes_{\beta}\ZZ$. In particular, $H\rtimes \ZZ \, \isom \, H'\rtimes \ZZ \nImp H\isom  H'$.
		\end{cor}

		\begin{proof}
			Let $K$ be any nontrivial finitely presented group. Consider $H=K\freeprod \ZZ$,
			which is also finitely presented, and $H'=\left( \Freeprod_{i\in \ZZ} K \right)
			\times \ZZ$, which is not finitely generated. Combining~\eqref{eq:*K semidirect Z =
				K*Z} and~\eqref{eq:rtimes times = times rtimes}, we get
			\begin{equation*}
				H\times \ZZ =
				\left(K \freeprod \ZZ \right) \times \ZZ \ \isom \,
				\left(\left( \Freeprod_{i\in \ZZ} K \right) \rtimes_\tau \ZZ \right) \times \ZZ \ \isom  \
				\left(\left( \Freeprod_{i\in \ZZ} K \right) \times \ZZ \right) \rtimes_T \ZZ
				=H' \rtimes_T \ZZ \,,
			\end{equation*}
			where $\tau \in \Aut (\Freeprod_{i\in \ZZ}\, K)$ is the automorphism
			\eqref{eq:translation automorphism} defined in~\cref{lem:*K semidirect Z = K*Z}, and $T
			\in \Aut(H')$ the corresponding one according to~\cref{lem:rtimes times = times
				rtimes}. The result follows taking $\alpha =\id_H$ and~$\beta =T$.
		\end{proof}

		So, there is considerable flexibility in describing cyclic extensions as semidirect
		products. Even fixing the base group, this flexibility persists within the possible
		defining automorphisms. For example,
		one can easily see that $H\rtimes_{\gamma} \ZZ\simeq H\times \ZZ$, for every inner
		automorphism $\gamma \in \Inn(H)$. A bit more generally, the following is a folklore
		lemma which is straightforward to prove (see~\cite{bogopolski_automorphism_2007}).

		\begin{lem}\label{lem:suf cond for isomorphic Z-extensions}
			Let $H$ be an arbitrary group, and let $\alpha, \beta \in \Aut (H)$. If $\beta =\gamma
			\xi \alpha^{\pm 1}\xi^{-1}$ for some $\gamma \in \Inn H$ and some $\xi \in \Aut (H)$,
			then $H\rtimes_{\alpha} \ZZ \isom H\rtimes_{\beta} \ZZ$. \qed
		\end{lem}

		The existence of such $\gamma \in \Inn H$ and $\xi \in \Aut (H)$ is exactly the same as
		$[\beta]$ being conjugate to $[\alpha]^{\pm 1}$ in $\Out (H)$. This condition turns out
		to have some protagonism along the rest of the paper, making convenient to have a
		general shorthand terminology for it.

		\begin{defn}\label{def:semi-conjugacy}
			Let $G$ be an arbitrary group. A pair of elements $g,h \in G$ are said to be
			\defin{semi-conjugate} if $g$ is conjugate to either $h$ or $h^{-1}$; we denote this situation by~$g\sim h^{\pm 1}$.
		\end{defn}

		With this terminology, \cref{lem:suf cond for isomorphic Z-extensions} states that, when the defining automorphisms $\alpha,\beta \in \Aut (H)$ are semi-conjugate in $\Out (H)$, then the corresponding semidirect products $H\rtimes_{\alpha} \ZZ $ and $H\rtimes_{\beta}\ZZ$
		are isomorphic. Note also the following necessary condition: by~\cref{prop:deranged characterization}, $\alpha$ is deranged if and only if $\betti{H\rtimes_{\alpha}\ZZ }=1$ so, in order for $H\rtimes_{\alpha}\ZZ$ and $H\rtimes_{\beta}\ZZ$ to be isomorphic, a necessary condition is that $\alpha$ and $\beta$ are either both simultaneously deranged, or both not deranged.

		Apart from this, not much is known in general about characterizing or deciding when two $\ZZ$-extensions of a given group are isomorphic. In~\cite{bogopolski_automorphism_2007}, Bogopolski--Martino--Ventura proved that, when the base group $H$ is free of rank $2$, the converse to~\cref{lem:suf cond for isomorphic Z-extensions} also holds, providing a quite neat characterization of isomorphism within the family of $F_2 \by \ZZ$ extensions and (using the decidability of the conjugacy problem in $\Out (F_2)$, see~\cite{bogopolski_classification_2000}) a positive solution to the isomorphism problem
		within this family of groups.

		\begin{thm}[Bogopolski--Martino--Ventura,
			\cite{bogopolski_automorphism_2007}]\label{F2} Let $\alpha, \beta \in \Aut (F_2)$.
			Then,
			\begin{equation}\label{eq:isomorphic F2-by-Z}
				F_2 \rtimes_\alpha \ZZ \isom F_2 \rtimes_\beta \ZZ \Biimp [\alpha] \text{ and } [\beta] \text{ are semi-conjugate in } \Out (F_2).
			\end{equation}
			In particular, the isomorphism problem within the family $F_2 \by \ZZ$ is decidable.
			\qed
		\end{thm}

		However, in this same paper, a counterexample was given (suggested by W.~Dicks) to see
		that this equivalence is not true for free groups of higher rank, where the situation is, in general, much more complicated. The example is the following: consider the free group
		of rank 3, $F_3 = \pres{a,b,c}{\,}$, and the automorphisms $\alpha\colon F_3 \to F_3$
		given by $a\mapsto b\mapsto c\mapsto b^{-1}ab^{-2}c^3$, and $\beta\colon F_3\to F_3$ by
		$a\mapsto b\mapsto c\mapsto a^{-1}b^2cb^{-1}$. It happens that $F_3 \rtimes_{\alpha}\ZZ
		\simeq F_3 \rtimes_{\beta} \ZZ$ (see~\cite{bogopolski_automorphism_2007} for details),
		while $\alpha$ and $\beta$ are not semi-conjugate in $\Out (F_3)$ because they abelianize
		to two $3\times 3$ matrices of determinants, respectively, 1 and -1. As far as we know,
		the isomorphism problem for $F_r \by \ZZ$ groups is open for $r\geq 3$.

		\medskip

		The goal of the present section is to prove that an equivalence like \eqref{eq:isomorphic F2-by-Z} still holds, but under kind of an orthogonal condition: rather than restricting the base group to be $F_2$, we will leave $H$ arbitrary finitely generated, and impose conditions on the defining automorphism. Note that such an equivalence reduces the isomorphism problem in the family of restricted extensions, to the conjugacy problem in the corresponding family of outer automorphisms of the base group (or even to a weaker problem, if semi-conjugacy is not algorithmically equivalent to conjugacy).

		This context strongly suggests defining the semi-conjugacy problem much in the same way
		that the standard conjugacy problem, and asking for the relationship between them. We
		state both problems together in order to make the comparison clear.

		\begin{defn} \label{def:semi-conjugacy problem}
			Let $\pres{X}{R}$ be a finite presentation for a group $G$. Then:
			\begin{itemize}
				\item \defin{Conjugacy Problem} for $G$ [\,$\CP(G)$\,]: given two words $u,v$ in
				      $X^{\pm}$, decide whether they represent conjugate elements in $G$ ($ u \sim_{G}
				      v$) or not.
				\item \defin{Semi-conjugacy Problem} for $G$ [\,$\SCP(G)$\,]: given two words $u,v$ in
				      $X^{\pm}$, decide whether they represent semi-conjugate elements in $G$ ($ u
				      \sim_{G} v^{\pm 1}$) or not.
			\end{itemize}
		\end{defn}

		\begin{qst}
			Is there a (finitely presented) group with decidable semi-conjugacy problem but undecidable conjugacy
			problem?
		\end{qst}

		This question looks quite tricky. Of course, if two elements $g,h\in G$ are not semi-conjugate, then they are not conjugate either. But if $g\sim h^{-1}$, it is not clear how this information can help, in general, to decide whether $g\sim h$ or not; in this sense the answer to the question seems reasonable to be negative. But, on the other hand, the two algorithmic problems are so close that it seems hard to construct a counterexample.

		\medskip

		In our case, the condition demanded for the defining automorphisms is derangedness (see~\ref{def:deranged}). The first observation is the following: suppose $H\rtimes_{\alpha}\ZZ \isom K\rtimes_{\beta}\ZZ$ for some groups $H$, $K$, and some deranged automorphisms $\alpha \in \Aut(H)$ and $\beta\in \Aut(K)$. Then, by construction, $H$ and $K$ are respectively, the unique normal subgroups with quotient isomorphic to $\ZZ$. Hence $H \isom K$ and, after expressing $\beta$ in terms of the generators of $H$, we can think that both $\alpha, \beta\in \Aut(H)$. The next step is to show that, under the derangedness condition, $H\rtimes_{\alpha}\ZZ \isom H\rtimes_{\beta}\ZZ$ implies that $[\alpha], [\beta]\in \Out(H)$ are semi-conjugate. To see this, we need to analyze how homomorphisms between unique $\ZZ$-extensions look like.

		%

		\begin{defn}\label{def:stable endomorphism}
			Let $G$ be a group, and $H$ a subgroup of $G$. An endomorphism $\Phi \in \End (G)$ is called \emph{$H$-stable} if it leaves $H$ invariant as a subgroup. The collection of all $H$-stable endomorphisms of $G$ form a submonoid denoted $\End_{H}(G) \leqslant \End (G)$. In a similar way, the collection of all $H$-stable automorphisms of~$G$ form a subgroup denoted $\Aut_{H}(G)\leqslant \Aut(G)$.
		\end{defn}

		A general description of the $H$-stable endomorphisms and automorphisms of infinite-cyclic extensions of $H$ follows.

		\begin{prop}\label{prop:isomorphism between G_deranged}
			Let $H$ be a group generated by $X=\{ x_i \mid i\in I\}$, and let $\alpha, \beta \in \Aut(H)$. Then, any homomorphism from $H\rtimes_{\alpha}\ZZ$ to $H\rtimes_{\beta}\ZZ$ mapping $H$ to $H$ is of the form
			\begin{equation}\label{eq:stable endos}
				\begin{array}{rcl}
					\Phi_{\epsilon,\phi,h_0} \colon H\rtimes_{\alpha}\ZZ & \to     & H\rtimes_{\beta}\ZZ, \\
					x_i                                                  & \mapsto & \phi(x_i)            \\
					t                                                    & \mapsto & h_0 \, t^{\epsilon}
				\end{array}
			\end{equation}
			where $\epsilon \in \ZZ$, $h_0\in H$, and $\phi \in \End(H)$ are such that
			$\gamma_{h_0}\beta^{\epsilon} \phi =\phi \alpha$.

			Furthermore, $\Phi_{\epsilon,\phi ,h_0}$ is an isomorphism if and only if $\epsilon
			=\pm 1$ and $\phi \in \Aut(H)$. Thus, $H$-stable automorphisms of
			$H\rtimes_{\alpha}\ZZ$ are precisely
			\begin{align}\label{eq:stable isos}
				\Aut_{H}(H\rtimes_{\alpha}\ZZ) & =\left\{ \, \Phi_{\epsilon,\phi,h_0} \mid \epsilon =\pm 1,\, h_0\in H,\, \phi \in \Aut(H) \text{ s.t. } \gamma_{h_0}\alpha^{\epsilon} \phi =\phi \alpha \, \right\} .
			\end{align}
		\end{prop}

		\begin{proof}
			Let $\Phi \colon H\rtimes_{\alpha}\ZZ \to H\rtimes_{\beta}\ZZ$ be a homomorphism
			leaving $H$ invariant, and let us denote by $\phi\colon H\to H$ its restriction to $H$.
			Write $\Phi(t) =h_0\, t^{\epsilon}$ for some $h_0\in H$ and $\epsilon \in \ZZ$.
			Applying $\Phi$ to both sides of the relation $tht^{-1} =\alpha (h)$ in the domain, we
			get
			\begin{equation*}
				h_0 \cdot \beta ^{\epsilon } \phi (h)\cdot h_0^{-1}=h_0 t^{\epsilon} \cdot \phi(h)\cdot t^{-\epsilon}h_0^{-1} =\Phi (tht^{-1}) =\Phi (\alpha(h))=\phi \alpha (h) \, ,
			\end{equation*}
			for all $h\in H$. Hence, $\gamma_{h_0}\beta^{\epsilon} \phi =\phi \alpha$ and $\Phi
			=\Phi_{\epsilon,\phi ,h_0}$ has the desired form.

			Assume now that $\Phi_{\epsilon,\phi ,h_0}$ is an isomorphism (in particular,
			$\phi\colon H\to H$ is injective). Then we must have $\epsilon =\pm 1$, otherwise $t$
			would not be in the image. On the other hand, since $H\normaleq H\rtimes_{\alpha}\ZZ$,
			we have that
			\begin{equation*}
				\phi(H)=\Phi_{\epsilon,\phi ,h_0}(H) \normaleq H\rtimes_{\beta}\ZZ
				=\Phi_{\epsilon,\phi ,h_0}(H\rtimes_{\alpha}\ZZ) =\langle \phi(H),\, h_0 t^{\epsilon} \rangle \, ,
			\end{equation*}
			and so, any element of $H\rtimes_{\beta }\ZZ$ can be written in the form $\phi(h)(h_0
			t^{\epsilon })^k$, for some $h\in H$ and $k\in \ZZ$; and it belongs to $H$ if and only
			if $k=0$. Thus, $\Phi(H)=H$ and $\phi \in \Aut(H)$. For the converse, it is clear that
			$\epsilon =\pm 1$ and $\phi \in \Aut(H)$ implies that $\Phi_{\epsilon,\phi ,h_0}$ is an
			isomorphism. The final statement follows immediately.
		\end{proof}

		Note that
		\cref{prop:deranged characterization} states precisely that $\End_{H} (H\rtimes_{\alpha }\ZZ)=\End (H\rtimes_{\alpha}\ZZ)$ if and only if $\alpha$ is deranged. This fact, together with the previous description provides a characterization of isomorphic deranged extensions in terms of semi-conjugacy.

		\begin{cor}\label{cor:isomorphic deranged extensions}
			Let $H$ and $K$ be two arbitrary groups, and let $\alpha\in \Aut(H)$ and $\beta\in
			\Aut(K)$ be two deranged automorphisms. Then,
			\begin{equation}\label{eq:isomorphic deranged extensions}
				H\rtimes_\alpha \ZZ \isom K\rtimes_\beta \ZZ \Biimp H\isom K \text{ and } [\alpha]\sim [\beta']^{\pm 1} \text{ in } \Out (H) \, ,
			\end{equation}
			where $\beta'=\psi^{-1}\beta\psi\in \Aut(H)$, and $\psi\colon H\to K$ is any
			isomorphism.
		\end{cor}

		\begin{proof}
			For any isomorphism $\psi\colon H\to K$, it is clear that $K\rtimes_{\beta}\ZZ
			=\psi(H)\rtimes_{\beta}\ZZ \simeq H\rtimes_{\psi^{-1}\beta\psi}\ZZ$. Hence, the
			statement is equivalent to saying
			\begin{equation*}
				H\rtimes_\alpha \ZZ \isom H\rtimes_\beta \ZZ \Biimp [\alpha]\sim [\beta]^{\pm 1} \text{ in } \Out (H) \, ,
			\end{equation*}
			for $\alpha, \beta\in \Aut(H)$. The implication to the left is a general fact
			(see~\cref{lem:suf cond for isomorphic Z-extensions}), and the implication to the right
			is a direct consequence of~\cref{prop:isomorphism between G_deranged}: since $\alpha$
			and $\beta$ are deranged, any isomorphism from $H\rtimes_\alpha \ZZ$ to $H\rtimes_\beta
			\ZZ$ must map $H$ to $H$ and so, must be of the form $\Phi_{\epsilon, \phi, h_0}$ for
			some $\epsilon =\pm 1$, $h_0\in H$, and $\phi \in \Aut(H)$ satisfying
			$\gamma_{h_0}\beta^{\epsilon} \phi =\phi \alpha$. Hence, $[\alpha]\sim [\beta]^{\pm 1}$
			in $\Out (H)$.
		\end{proof}

		We are now ready to prove the main result in this section: for any family $\HH$ of finitely presented groups with decidable isomorphism problem, we characterize when the family $\fpa{!\HH \by \ZZ}$ has again decidable isomorphism problem, in terms of a certain variation of the conjugacy problem for outer automorphisms of groups in $\HH$.

		\medskip

		Note that \cref{cor:isomorphic deranged extensions} clearly insinuates a link between
		the isomorphism problem for deranged extensions, and the semi-conjugacy problem for deranged outer automorphisms of the base group. However, there is a subtlety at this point: the supposed algorithm solving the isomorphism problem will receive the input (the compared groups) as finite presentations of the $\ZZ$-extensions. From those, we know how to compute suitable base groups $H,K$, and automorphisms $\alpha, \beta$ (see \cref{prop:enumerate standard presentations of fg-by-Z}), but this last ones are \emph{given by images of the generators in the starting presentations}, and not as words in some presentation of the corresponding automorphism groups, which would be the appropriate inputs for the standard conjugacy problem there.

		So, in general, one must distinguish between these two close but not necessarily identical situations. As before, we state both problems together to emphasize the difference between them.

		%

		\begin{defn}
			Let $\pres{X}{R}$ be a presentation for a group $G$, $\pres{Y}{S}$ a presentation for $\Aut(G)$, and assume $|X|<\infty$. Then:
			\begin{itemize}
				\item \defin{(Standard) conjugacy problem} for $\Aut(G)$ [\,$\CP(\Aut(G))$\,]: given
				      two automorphisms $\alpha, \beta \in \Aut(G)$ \emph{as words in the presentation
				      	of $\Aut(G)$}, decide whether $\alpha$ and $\beta$ are conjugate to each other in
				      $\Aut(G)$.
				\item \defin{$\Aut$-conjugacy problem} for $G$
				      [\,$\CP_{\!\scriptscriptstyle{G}}(\Aut(G))$\,]: given two automorphisms $\alpha,
				      \beta \in \Aut(G)$ \emph{by images of (the finitely many) generators $X$}, decide whether $\alpha$ and $\beta$ are conjugate to each other in~$\Aut(G)$.
			\end{itemize}
			Similarly, we define the \emph{$\Out$-conjugacy problem} [\,$\DCP{G}(\Out(G))$\,], the
			\emph{$\Aut$-semi-conjugacy problem} [\,$\DSCP{G}(\Aut(G))$\,], and the
			\emph{$\Out$-semi-conjugacy problem} [\,$\DSCP{G}(\Out(G))$\,] for $G$ (in contrast with
			the standard $\CP(\Out(G))$, $\SCP(\Aut(G))$, and $\SCP(\Out(G))$).
		\end{defn}

		Note that, in general, these pairs of problems are similar but not identical: from the algorithmic point of view it could be very different to have an automorphism of $G$ given as the collection of images of a finite set of generators of $G$, or as a word (composition of generators for $\Aut(G)$). Consider, for example,  the Baumslag--Solitar group $G=\BS(2,4)$, which is finitely generated, but whose automorphism group $\Aut(G)$ is known to be not finitely generated (see~\cite{collins_automorphisms_1983}).

		However, knowing in advance a finite set of generators for $\Aut(G)$ (respectively,
		$\Out(G)$) as images of generators of $G$, these two kinds of problems turn out to be
		equivalent.

		\begin{prop}
			Let $\pres{X}{R}$ be a presentation for a group $G$, $X=\{x_1, \ldots ,x_n\}$, and $\{u_{i,j} \mid i=1, \ldots ,n,\,\, j=1,\ldots, N \}$ a finite set of words in $X^{\pm}$ such that $\{\alpha_j\colon x_i\mapsto u_{i,j} \mid j=1,\ldots, N \}$ is a well defined finite family of automorphisms generating~$\Aut(G)$. Then,
			\begin{equation} \label{eq:AutCP <-> CPAut}
				\DCP{G}(\Aut(G)) \text{ is decidable} \Biimp \CP(\Aut(G)) \text{ is decidable\,}.
			\end{equation}
			The same is true replacing conjugacy by semi-conjugacy, and/or $\Aut$ by $\Out$.
		\end{prop}

		\begin{proof}
			Suppose that $G$ has decidable $\Aut$-conjugacy problem. Given two automorphisms
			$\alpha, \beta \in \Aut(G)$ as words on the $\alpha_i$'s, say $\alpha=a(\alpha_1,
			\ldots ,\alpha_N)$ and $\beta=b(\alpha_1, \ldots ,\alpha_N)$, we can compute the
			corresponding compositions of $\alpha_j$'s and obtain explicit expressions for $\alpha
			(x_i)$ and $\beta(x_i)$ in terms of $X$, for $i=1,\ldots ,n$. Now, applying the
			solution to the $\Aut$-conjugacy problem for $G$ we decide whether $\alpha$ and $\beta$
			are conjugate to each other in $\Aut(G)$.

			Conversely, suppose $\Aut(G)$ has decidable conjugacy problem, and we are given two
			automorphisms $\alpha, \beta \in \Aut(G)$ by the images of the $x_i$'s, say
			$\alpha(x_i)$ and $\beta(x_i)$, $i=1,\ldots ,n$. We will express $\alpha$ and $\beta$
			as compositions of the $\alpha_j$'s, and then apply the assumed solution to the conjugacy problem for $\Aut(G)$ to decide whether $\alpha$ and $\beta$ are conjugate to each other, or not. We  can do this by a brute force enumeration of all possible formal reduced words $w$ on $\alpha_1, \ldots ,\alpha_N$ and, for each one, computing the tuple $(w(x_1), \ldots ,w(x_n))$ and trying to check whether it equals $(\alpha(x_1),
			\ldots ,\alpha(x_n))$, or $(\beta(x_1), \ldots ,\beta(x_n))$ (following a brute force enumeration of the normal closure $\normalcl{R}$, like in the proof of \cref{prop:enumerate standard presentations of fg-by-Z}).


			The proofs of the other versions of the statement are completely analogous. For the conjugacy problems we need to add another brute force search layer enumerating all possible conjugators; we leave details to the reader.
		\end{proof}

		After this proposition we can prove the main result in this section.

		\begin{thm}\label{thm:main-iso}
			Let $\HH$ be a family of finitely presented groups with decidable isomorphism problem.
			Then, the isomorphism problem of $!\HH \by \ZZ$ is decidable if and only if the
			$\Out_{\der}$-semi-conjugacy problem of $H$ is decidable for every $H$ in $\HH$; \ie
			\begin{equation*}
				\IP(!\HH \by \ZZ) \text{ decidable} \Biimp \DSCP{H}(\Out_{\der}(H)) \text{ decidable, }\forall H \in \HH \, .
			\end{equation*}
		\end{thm}

		\begin{proof}

			Suppose that every $H \in \HH$ has decidable $\Out$-semi-conjugacy problem for deranged
			inputs. Given finite presentations of two groups $G$ and $G'$ in $!\HH \by \ZZ$, we
			run~\cref{prop:enumerate standard presentations of fg-by-Z} to compute standard
			presentations for them, and extract finite presentations for base groups and defining
			automorphisms (say $H$ and $\alpha\in \Aut(H)$ for $G$, and $K$ and $\beta\in \Aut(K)$
			for $G'$, respectively). We have $G=H\rtimes_\alpha \ZZ$ and $G'=K\rtimes_\beta \ZZ$
			and, by hypotheses, $\alpha$ and $\beta$ are deranged. Furthermore, since
			$\betti{G}=\betti{G'}=1$, $H$ and $K$ are the unique normal subgroups of $G$ and $G'$,
			respectively, with quotient $\ZZ$; hence, $H,K\in \HH$.

			Now we apply the isomorphism problem within $\HH$ to the obtained presentations for $H$
			and $K$, and decide whether they are isomorphic as groups. If $H\not\simeq K$ then,
			by~\cref{cor:isomorphic deranged extensions}, $G\not\isom G'$ and we are done.
			Otherwise, we construct an explicit isomorphism $\psi\colon H\to K$ (by a brute force
			search procedure like the ones above), we compute $\beta'=\psi^{-1}\beta\psi\in
			\Aut(H)$, and we apply our solution to the $\Out$-semi-conjugacy problem for $H\in \HH$
			to the inputs $\alpha$ and $\beta'$, (which are deranged, by construction). The output
			on whether $[\alpha]$ and $[\beta']$ are or are not semi-conjugate  in $\Out (H)$ is the
			final answer we are looking for (again by~\cref{cor:isomorphic deranged extensions}).

			For the converse, assume that the isomorphism problem is decidable in the family
			$!\HH \by \ZZ$, and fix a finite presentation $\pres{X}{R}$ for a group $H \in
			\HH$. Given two deranged automorphisms $\alpha, \beta\in \Out(H)$ via images of the
			generators $x_i\in X$, build the corresponding standard presentations for
			$H\rtimes_{\alpha}\ZZ$ and $H\rtimes_{\beta}\ZZ$ (which are groups in $!\HH \by
			\ZZ$, by construction) and apply the assumed solution to the isomorphism problem for
			this family to decide whether they are isomorphic or not. By~\cref{cor:isomorphic
				deranged extensions}, the answer is affirmative if and only if $[\alpha]$ and $[\beta]$
			are semi-conjugate in $\Out(H)$.
		\end{proof}

		We apply now \cref{thm:main-iso} to special families of groups with decidable isomorphism problem. Some of these corollaries are already known in the literature; other methods provide alternative approaches.

		Taking $\HH$ to be a single group $H$, we get the following result.

		\begin{cor}
			Let $H$ be a finitely presented group. Then the isomorphism problem is decidable within
			the family $!H \by \ZZ$ if and only if $H$ has decidable $\Out$-semi-conjugacy problem
			for deranged inputs.
			In particular,
			if $|\Out(H)|<\infty$, then $!H \by \ZZ$ has decidable isomorphism problem. \qed
		\end{cor}

		%
		%
		%

		Taking $\HH$ to be the families of finite, finitely generated abelian, or polycyclic
		groups, \cref{thm:main-iso} provides well-known results, since the obtained extensions
		turn out to be subfamilies of that of virtually-polycyclic groups for which the
		isomorphism problem is known to be decidable (see~\cite{segal_decidable_1990}).

		\medskip

		For the family of Braid groups $\BB =\propsty{braid} =\{ B_n \mid n\geqslant 2\}$, the specially simple structure of its outer automorphism group allows us to state the isomorphism problem within the family $\BB\by \ZZ$.

		\begin{cor}
			The isomorphism problem is decidable within the family $\BB\by \ZZ$.
		\end{cor}

		\begin{proof}
			It is well known that, for every $n\geq 2$, $\Out (B_n) =\{ \overline{\id}, \overline{\iota}\}$, where $\iota\colon B_n \to B_n$ is the automorphism given by $\sigma_i \mapsto \sigma_i^{-1}$ (see~\cite{dyer_automorphism_1981}). Then, (since automorphisms in the same inner class induce the same $\ZZ$-extension, by \cref{lem:suf cond for isomorphic Z-extensions}), we have
			\begin{equation} \label{eq:braid-by-Z isomorphism classes}
				\BB\by \ZZ =\{ B_n \times \ZZ , n \geq 2 \} \cup \{ B_n \rtimes_{\iota} \ZZ, n \geq 2 \} \, .
			\end{equation}

			We claim that they are all pairwise nonisomorphic. Those in the first term of
			the union cannot be isomorphic to those in the second one because $\id$ is not deranged, while $\iota$ is (in other words, $\betti{B_n \times \ZZ}=2$ while $\betti{B_n \rtimes_{\iota} \ZZ }=1$). Two deranged extensions $B_n \rtimes_{\iota} \ZZ$ and $B_m \rtimes_{\iota} \ZZ$ can only be isomorphic if their base groups are, and this happens only when $n=m$ (this can be seen, for example, by observing that the center $\textrm{Z}(B_n)$ is generated by the full twist $(\sigma_1 \sigma_2 \cdots \sigma_{n-1})^n$ and so, the abelianization of $B_n /\textrm{Z}(B_n)$ is cyclic of order~$n(n-1)$). Finally, the same argument shows that $B_n \times \ZZ \simeq B_m \times \ZZ$ if and only if $n=m$.

			Thus, the isomorphism problem within $\BB\by \ZZ$ is decidable: given two finite presentations of groups $G_1$ and $G_2$ in $\BB\by \ZZ$, explore the two trees of Tietze transformations until finding standard presentations for them, \ie until recognizing their number of strands, say $n$ and $m$. Now $G_1\simeq G_2$ if and only if $\betti{G_1 }=\betti{G_2 }$ and $n=m$.
		\end{proof}

		Finally, let us consider the case of finitely generated free groups, $\FF =
		\fga{\propsty{free}} =\{ F_n \mid n\geqslant 0\}$. To start with, the isomorphism
		problem for $\FF$ is decidable like in the case of Braid groups (since $F_n
		\simeq F_m \, \Leftrightarrow \, n=m$). A solution to the conjugacy problem in
		$\Out(F_n)$ was announced by M.~Lustig in the
		preprints~\cite{lustig_structure_2000,lustig_structure_2001}. Although this project is
		not completed (and there is no published version yet), it is believed that $\Out(F_n)$
		has decidable conjugacy problem. However, at this moment we can only say to have firm
		complete solutions for some classes of outer automorphisms:
		\begin{enumerate}[(i)]
			\item the case of rank 2 is easily decidable because $\Out(F_2)\simeq \GL_2(\ZZ)$;
			\item for finite-order elements of $\Out(F_n)$ an algorithm to solve the conjugacy
			      problem follows from results of S. Krsti\'c (see~\cite{krstic_actions_1989});
			\item J. Los and, independently, Z. Sela solved the conjugacy problem in
			      $\Out(F_r)$ for irreducible inputs,
			      see~\cite{sela_isomorphism_1995,los_conjugacy_1996,lustig_conjugacy_2007};
			\item for Dehn twist automorphisms, the conjugacy problem has been solved by
			      Cohen--Lustig, see~\cite{cohen_conjugacy_1999};
			\item finally, Krsti\'c--Lustig--Vogtmann solved the conjugacy problem in
			      $\Out(F_n)$ for linearly growing automorphisms, \ie for roots of Dehn twists,
			      see~\cite{krstic_equivariant_2001}.
		\end{enumerate}

		If the conjugacy problem in $\Out(F_n)$ were decidable in general, we could
		deduce from~\cref{thm:main-iso} that the isomorphism problem for the family
		$!\fga{\propsty{free}} \by \ZZ$ is decidable as well. By the moment, we can only
		restrict our attention to the above mentioned subsets of $\Out(F_n)$, where the
		conjugacy problem is firmly known to be decidable, and we obtain the isomorphism problem
		for the corresponding subfamilies~(see the proof of~\cref{thm:main-iso}).

		\begin{cor}
			If the conjugacy problem for $\Out(F_n)$ is decidable, then the isomorphism problem
			within the family $!\fga{\propsty{free}} \by \ZZ$ is also decidable. \qed
		\end{cor}

		\begin{cor}
			The isomorphism problem within the following families is decidable:
			\begin{enumerate}[(i)]
				\item $!F_2 \by \ZZ$;
				\item $\{ F_n\rtimes_{\alpha}\ZZ \mid \alpha\in\Out(F_n) \text{ deranged and finite
					      order} \,\}$;
					\item $\{ F_n\rtimes_{\alpha}\ZZ \mid \alpha\in\Out(F_n) \text{ deranged and
						      irreducible}\,\}$;
						\item $\{ F_n\rtimes_{\alpha}\ZZ \mid \alpha\in\Out(F_n) \text{ deranged and linearly
							      growing} \,\}$. \qed
						\end{enumerate}
					\end{cor}

					It is worth mentioning that our approach is somehow opposite to that taken by Dahmani
					in~\cite{dahmani_suspensions_2016}. In this interesting preprint the author solves the
					conjugacy problem for atoroidal automorphisms of $F_n$. An automorphism $\alpha \in
					\Out(F_n)$ is \defin{atoroidal} if no proper power of $\alpha$ fixes any nontrivial
					conjugacy class (note that this notion is similar in spirit to our notion of
					derangedness, though they do not coincide). Brinkmann proved
					in~\cite{brinkmann_hyperbolic_2000} that $F_n \rtimes_{\alpha} \ZZ$ is hyperbolic if
					and only if $\alpha$ is atoroidal. And $\alpha, \beta \in \Out(F_n)$ are conjugate to
					each other if and only if $F_n \rtimes_{\alpha}\ZZ$ is isomorphic to $F_n
					\rtimes_{\beta}\ZZ$ with an automorphism mapping $F_n$ to $F_n$, and $t$ to an element
					of the form $wt^1$ (\ie with an stable and positive automorphism in our language, see
					the proof of~\cref{cor:isomorphic deranged extensions}). Then, Dahmani uses a variation
					of the celebrated solution to the general isomorphism problem for hyperbolic groups
					(see~\cite{sela_isomorphism_1995,dahmani_isomorphism_2011}) to determine whether $F_n
					\rtimes_{\alpha}\ZZ$ and $F_n \rtimes_{\beta}\ZZ$ are isomorphic through an isomorphism
					of the above type, and so deciding whether the atoroidal automorphisms $\alpha$ and $\beta$ are conjugated to each other in $\Out(F_n)$. Our approach has been the opposite: we have used the conjugacy problem in $\Out(F_n)$ (more precisely, those particular cases where it is known to be decidable) to solve the isomorphism problem in the corresponding
					families of unique $\ZZ$-extensions.

					\bigskip

					\subsection*{Acknowledgements}

					We thank Ha Lam for interesting conversations during an early stage of the development of this paper. We also thank Maurice Chiodo for suggesting a nice trick used in \cref{thm:fg-by-Z is undecidable};
					and Conchita Martínez for pointing out the counterexample used in~\cref{prop: fp-by-Z < [fg-by-Z]fp}.

					\medskip

					The second named author thanks the support of
					\emph{Universitat Polit\`{e}cnica de Catalunya} through the PhD grant number 81--727.
					The third named author was partially supported by a PSC-CUNY grant from the CUNY research foundation,  the City Tech foundation, the ONR (Office of Naval Research) grant N000141210758 and N00014-15-1-2164, and AAAS (American Association for The Advancement of Science) grant 71527-0001.
					Finally, the second and fourth named authors acknowledge partial support from the Spanish Government through grant number MTM2014-54896-P.

					\bigskip

					\renewcommand{\bibfont}{\normalfont\small}
					\bibitemsep = 1ex \printbibliography


\end{document}